\documentclass{article} 
\usepackage{iclr2021_conference,times}

\usepackage{amsmath,amsfonts,bm}

\def\1{\bm{1}}

\DeclareMathAlphabet{\mathsfit}{\encodingdefault}{\sfdefault}{m}{sl}
\SetMathAlphabet{\mathsfit}{bold}{\encodingdefault}{\sfdefault}{bx}{n}

\newcommand{\R}{\mathbb{R}}

\DeclareMathOperator*{\argmin}{arg\,min}

\usepackage{hyperref}
\usepackage{url}

\usepackage{booktabs}       
\usepackage{amsfonts}       
\usepackage{nicefrac}       
\usepackage{microtype}      
\usepackage{graphicx}
\usepackage{amsmath,amsfonts,amssymb,multirow,array} 
\usepackage{xcolor}
\usepackage{hyperref}
\hypersetup{colorlinks}
\usepackage{fec}
\usepackage{dpr}
\usepackage{mathtools}
\usepackage{algorithm}
\usepackage{enumerate}
\usepackage{subcaption}
\usepackage[noend]{algpseudocode}
\usepackage{makecell}
\usepackage{comment}
\usepackage{caption}
\usepackage{enumitem}
\usepackage{amsthm}
\usepackage{arydshln}
\usepackage{bm}
\usepackage{siunitx,tabularx,ragged2e,booktabs}
\usepackage{wrapfig}
\usepackage{subcaption}

\algnewcommand{\algorithmicforeach}{\textbf{for each}}
\algdef{SE}[FOR]{ForEach}{EndForEach}[1]
  {\algorithmicforeach\ #1\ \algorithmicdo}
  {\algorithmicend\ \algorithmicforeach}

\newcommand{\rda}{\text{RDA}}
\newcommand{\proxsvrg}{\text{Prox-SVRG}}
\newcommand{\proxsg}{\text{Prox-SG}}
\newcommand{\proxspider}{\text{Prox-Spider}}
\newcommand{\saga}{\text{SAGA}}
\newcommand{\algname}{Half-Space Stochastic Projected Gradient}

\newcommand{\algacro}{HSPG{}}

\newcommand{\cifar}{\text{CIFAR10}}
\newcommand{\fashionmnist}{\text{Fashion-MNIST}}
\newcommand{\resnet}{\text{ResNet18}}
\newcommand{\vgg}{\text{VGG16}}
\newcommand{\mobilenet}{\text{MobileNetV1}}
\newcommand{\myeq}{\stackrel{\mathclap{\normalfont\mbox{\small def}}}{=}}

\newcommand{\prox}{\text{Prox}}
\newcommand{\proxsgstep}{\text{Prox-SG Step}}
\newcommand{\halfspacestep}{\text{Half-Space Step}}

\newcommand{\ie}{\textit{i.e.}}
\newcommand{\eg}{\textit{e.g.}}

\newtheorem{theorem}{Theorem}
\newtheorem{lemma}{Lemma}
\newtheorem{corollary}{Corollary}
\newtheorem{assumption}{Assumption}
\newtheorem{proposition}{Proposition}

\def \B {\mathcal{B}}

\def \R {\mathbb{R}}

\def \I {\mathcal{I}}
\def \S {\mathcal{S}}
\def \P {\mathcal{P}}

\def \G {\mathcal{G}}
\def \O {\mathcal{O}}

\title{Half-Space Proximal Stochastic Gradient Method for Group-Sparsity Regularized Problem}

\author{\textbf{Tianyi Chen} \\
Microsoft\\
\texttt{tiachen@microsoft.com} \\
\And
\textbf{Guanyi Wang} \\
Georgia Institute of Technology \\
\texttt{gwang93@gatech.edu} \\
\And
\textbf{Tianyu Ding} \\
Johns Hopkins University \\
\texttt{tding1@jhu.edu}
\AND
\textbf{Bo Ji} \\
Zhejiang University \\
\texttt{jibo27@zju.edu.cn}
\And
\textbf{Sheng Yi} \\
Microsoft\\
\texttt{shengyi@microsoft.com}\\
\And
\textbf{Zhihui Zhu} \\
University of Denver\\
\texttt{zhihui.zhu@du.edu}
}

\iclrfinalcopy 
\begin{document}

\maketitle

\begin{abstract}
Optimizing with group sparsity is significant in enhancing model interpretability in machining learning applications, \eg, feature selection, compressed sensing and model compression. However, for large-scale stochastic training problems, 
effective group sparsity exploration are typically hard to achieve. 
Particularly, the state-of-the-art stochastic optimization algorithms 
usually generate merely dense solutions.
To overcome this shortage, we propose a stochastic method---Half-space Stochastic Projected Gradient (HSPG) method
to 
search solutions of high group sparsity
while maintain the convergence. Initialized by a simple Prox-SG Step, the \algacro{} method 
relies on a novel~\halfspacestep{} to substantially boost the sparsity level.
Numerically, \algacro{} demonstrates its superiority in 
deep neural networks, \eg,~\vgg{},~\resnet{} and~\mobilenet, by computing solutions of  higher group sparsity, competitive objective values and generalization accuracy.
\end{abstract}

\section{Introduction}

In many recent machine learning optimization tasks, researchers not only focus on finding solutions with small prediction/generalization error
but also concentrate on improving the interpretation of model by filtering out redundant parameters and achieving slimmer model architectures.  One technique to achieve the above goal is by augmenting the sparsity-inducing regularization terms to the raw objective functions to generate sparse solutions (including numerous zero elements). The popular $\ell_1$-regularization promotes the sparsity of solutions by element-wise penalizing the optimization variables. However, in many practical applications, there exist additional constraints on variables such that the zero coefficients are often not randomly distributed but tend to be clustered into varying more sophisticated sparsity structures, \eg, disjoint and overlapping groups and hierarchy \citep{yuan2006model,huang2010benefit,huang2009learning}.
As the most important and natural form of structured sparsity, the disjoint group-sparsity regularization, which assumes the pre-specified disjoint blocks of variables 
are selected (non-zero variables) or ignored (zero variables) simultaneously~\citep{bach2012structured}, serves as a momentous role in general structured sparsity learning tasks since other instances such as overlapping group and hierarchical sparsity are typically solved by converting into the equivalent disjoint group versions via introducing latent variables~\citep{bach2012structured}, and has found numerous applications in computer vision \citep{elhamifar2012see}, signal processing \citep{chen2014group}, medical imaging \citep{liu2018mri}, and deep learning~\citep{scardapane2017group}, especially on the model compression of deep neural networks, where the group sparsity\footnote{Group sparsity is defined as \# of zero groups, where a zero group means all its variables are exact zeros.} is leveraged to remove redundant entire hidden structures directly.

\textbf{Problem Setting.} We study the disjoint group sparsity regularization problem which can be typically formulated as the mixed $\ell_1/\ell_p$-regularization problem, and pay special attention to the most popular and widely used instance $p$ as $2$~\citep{bach2012structured,el2018combinatorial},
\begin{equation}\label{prob.x}
\minimize{\bm{x}\in \R^n}\ \Big\{\Psi(\bm{x})\ \myeq\ f(\bm{x})+\lambda \Omega(\bm{x})={\frac{1}{N}\sum_{i=1}^N f_i(\bm{x})}+\lambda{\sum_{g\in\mathcal{G}}\norm{[\bm{x}]_g}}\Big\},
\end{equation}
where $\lambda>0$ is a weighting factor, $\norm{\cdot}$ denotes $\ell_2$-norm, $f(\bm{x})$ is the average of numerous $N$ continuously differentiable instance functions $f_i : \R^n\rightarrow\R$, such as the loss functions measuring the deviation from the observations in various data fitting problems, $\Omega(\bm{x})$ is the so-called mixed $\ell_1/\ell_2$ norm,  $\mathcal{G}$ is a prescribed fixed partition of index set $\mathcal{I}=\{1,2,\cdots, n\}$, wherein each component $g\in\mathcal{G}$ indexes a group of variables upon the perspective of applications. Theoretically, a larger $\lambda$ typically results in a higher group sparsity while sacrifices more on the bias of model estimation, hence $\lambda$ needs to be carefully fine-tuned to achieve both low $f$ and high group-sparse solutions.

\textbf{Literature Review.\;} Problem~(\ref{prob.x}) has been well studied in deterministic optimization with various algorithms that are capable of returning solutions with both low objective value and high group sparsity under proper $\lambda$~\citep{yuan2006model,roth2008group,huang2011learning,ndiaye2017gap}. Proximal methods are classical approaches to solve the structured non-smooth optimization~\eqref{prob.x}, including the popular proximal gradient method (Prox-FG) which only uses the first-order derivative information. When $N$ is huge, stochastic methods become ubiquitous to operate on a small subset to avoid the costly evaluation over all instances in deterministic methods for large-scale problems. Proximal stochastic gradient method~(\proxsg)~\citep{duchi2009efficient} is the natural stochastic extension of Prox-FG. Regularized dual-averaging method (\rda)~\citep{xiao2010dual,yang2010online} is proposed by extending the dual averaging scheme in~\citep{nesterov2009primal}. To improve the convergence rate, 
there exists a set of incremental gradient methods inspired by  SAG~\citep{roux2012stochastic} to utilizes the average of accumulated past gradients. For example, 
proximal stochastic variance-reduced gradient method (\proxsvrg{})~\citep{xiao2014proximal} and proximal spider (\proxspider)~\citep{zhang2019multi} are developed to adopt multi-stage
schemes based on the well-known variance reduction technique SVRG proposed in~\citep{johnson2013accelerating} and Spider developed in~\citep{fang2018spider} respectively.  \saga{}~\citep{defazio2014saga} stands as the midpoint between SAG and Prox-SVRG.


Compared to deterministic methods, the studies of mixed $\ell_1/\ell_2$-regularization~(\ref{prob.x}) in stochastic field become somewhat rare and limited. \proxsg{},~\rda{},~\proxsvrg{}, Prox-Spider and~\saga{} are valuable state-of-the-art stochastic algorithms for solving problem~(\ref{prob.x}) but with apparent weakness. Particularly, these existing stochastic algorithms typically meet difficulties to achieve both decent convergence and effective group sparsity identification simultaneously (e.g., small function values but merely dense solutions), because of the randomness and the limited sparsity-promotion mechanisms. In depth,~\proxsg{},~\rda{},~\proxsvrg{},~\proxspider{} and~\saga{} derive from proximal gradient method to utilize the proximal operator to produce group of zero variables. Such operator is generic to extensive non-smooth problems, consequently perhaps not sufficiently insightful if the target problems possess certain properties,~\eg, the group sparsity structure as problem~(\ref{prob.x}). In fact, in convex setting, the proximal operator suffers from variance of gradient estimate; and in non-convex setting, especially deep learning, the discreet step size (learning rate) further deteriorates its effectiveness on the group sparsity promotion, as will show in Section~\ref{sec.algorithm} that the projection region vanishes rapidly except~\rda{}. \rda{} has superiority on finding manifold structure to others~\citep{lee2012manifold}, but inferiority on the objective convergence. Besides, the variance reduction techniques are typically required to measure over a huge mini-batch data points in both theory and practice which is probably prohibitive for large-scale problems, and have been observed as sometimes noneffective for deep learning applications~\citep{defazio2019ineffectiveness}. On the other hand, to introduce sparsity, there exist heuristic weight pruning methods~\citep{li2016pruning,luo2017thinet}, whereas they commonly do not equip with theoretical guarantee, so that easily diverge and hurt generalization accuracy. 

\textbf{Our Contributions.}  \algname{} (\algacro) method overcomes the limitations of the existing stochastic  algorithms on the group sparsity identification, while maintains comparable convergence characteristics. While the main-stream works on (group) sparsity have focused on using proximal operators of regularization, our method is unique and fresh in enforcing group sparsity more effectively by leveraging half-space structure and is well supported by the theoretical analysis and empirical evaluations. We now summarize our contributions as follows.
\vspace{-0.24cm}
\begin{itemize}[leftmargin=*]
	\item \emph{Algorithmic Design:} We propose the \algacro{} to solve the disjoint group sparsity regularized problem as~(\ref{prob.x}). Initialized with a Prox-SG Step for seeking a close-enough but perhaps dense solution estimate, the algorithmic framework relies on a novel~\halfspacestep{} to exploit group sparse patterns. We delicately design the Half-Space Step with the following main features: \textit{(i)} it utilizes previous iterate as the normal direction to construct a reduced space  consisting of a set of half-spaces and the origin; \textit{(ii)} a new group projection operator maps groups of variables onto zero if they fall out of the constructed reduced space to identify group sparsity considerably more effectively than the proximal operator; and \textit{(iii)} with proper step size, the Half-Space Step enjoys the sufficient decrease property, and achieves progress to optimum in both theory and practice.
	
	\item \emph{Theoretical Guarantee:} We provide the convergence guarantees of~\algacro{}.  Moreover, we prove \algacro{} has looser requirements to identify the sparsity pattern than Prox-SG, revealing its superiority on the group sparsity exploration. Particularly, for the sparsity pattern identification, the required distance to the optimal solution $\bm{x}^*$ of~\algacro{} is better than  the distance required by~\proxsg{}.

	\item \emph{Numerical Experiments:} Experimentally, \algacro{} outperforms the state-of-the-art methods in the aspect of the group sparsity exploration, and achieves competitive objective value convergence and runtime in both convex and non-convex problems. In the popular deep learning tasks,  \algacro{} usually computes the solutions with multiple times higher group sparsity and similar generalization performance on unseen testing data than those generated by the competitors, which may be further used to construct smaller and more efficient network architectures. 
	
\end{itemize}

\section{The \algacro\ method }\label{sec.algorithm}
We state the~\algname{} (\algacro) method in Algorithm~\ref{alg:main.x.outline}. In general, it contains two stages: Initialization Stage and Group-Sparsity Stage. The first Initialization Stage employs~\proxsgstep{} (Algorithm~\ref{alg:main.x.prox_sg_step}) to search for a close-enough but usually non-sparse solution estimate. Then the second and fundamental stage proceeds \halfspacestep{}  (Algorithm~\ref{alg:main.x.halfspacestep}) started with the non-sparse solution estimate to effectively exploit the group sparsity within a sequence of reduced spaces, and converges to the group-sparse solutions with theoretical convergence property.

\begin{algorithm}[h!]
	\caption{Outline of \algacro{} for solving \eqref{prob.x}.}
	\label{alg:main.x.outline}
	\begin{algorithmic}[1]
		\State \textbf{Input:} $x_0\in\mathbb{R}^n$, $ \alpha_0\in(0,1), \epsilon\in [0,1)$, and $ N_\mathcal{P}\in \mathbb{Z}^+ $.
		\For{$k = 0,1,2,\dots$ } 
		\If{$k< N_\mathcal{P}$} \label{line:switch_prox_sg_step}
		\State Compute $x_{k+1}\leftarrow \text{Prox-SG}(x_k,\alpha_k)$ by Algorithm~\ref{alg:main.x.prox_sg_step}. 
		\Else{}
		\State Compute $x_{k+1}\leftarrow\text{Half-Space}(x_k,\alpha_k, \epsilon)$ 
		by Algorithm \ref{alg:main.x.halfspacestep}. \hspace{0.3in}
		\EndIf
		\State Update $\alpha_{k+1}$. 
		\EndFor
	\end{algorithmic}
\end{algorithm}

\begin{algorithm}[h]
	\caption{Prox-SG Step.}
	\label{alg:main.x.prox_sg_step}
	\begin{algorithmic}[1]
		\State \textbf{Input:} Current iterate $x_k$, and step size $ \alpha_k$. 
		\State Compute the stochastic gradient of $f$ on mini-batch $ \mathcal{B}_k $
		\begin{equation}
		\nabla f_{\mathcal{B}_k}(x_k)\leftarrow\frac{1}{|\mathcal{B}_k|}\sum_{i\in \mathcal{B}_k}\Grad f_i(x_k).
		\end{equation}\label{line:g_t_estimate_prox_sg}
		\State \textbf{Return }
		$x_{k+1}\leftarrow\prox_{\alpha_k\lambda\Omega(\cdot)}\left(x_k-\alpha_k\nabla f_{\mathcal{B}_k}(x_k)\right)$ \label{line:prox}.
	\end{algorithmic}
\end{algorithm}

\vspace{-.15in}
\paragraph{Initialization Stage.}
The Initialization Stage performs the vanilla proximal stochastic gradient method (Prox-SG, Algorithm~\ref{alg:main.x.prox_sg_step}) to approach the solution of~(\ref{prob.x}). At $k$th iteration, a mini-batch $\mathcal{B}_k$ is sampled to generate an unbiased estimator of the full gradient of $f$ (line~\ref{line:g_t_estimate_prox_sg}, Algorithm~\ref{alg:main.x.prox_sg_step}) to compute a trial iterate $\widehat{x}_{k+1}:=x_k-\alpha_k \Grad f_{\mathcal{B}_k}(x_k)$, where $\alpha_k$ is the step size, and $f_{\B_k}$ is the average of the instance functions $f_i$ cross $\B_k$. The next iterate $x_{k + 1}$ is then updated based on the proximal mapping 
\begin{equation}\label{eq:proxmapping}
\begin{split}
x_{k+1}&=\prox_{\alpha_k\lambda\Omega(\cdot)}(\hat{x}_{k+1})=\argmin_{x\in \mathbb{R}^{n}}\ \frac{1}{2\alpha_k}\norm{x-\hat{x}_{k+1}}^2+ \lambda\Omega(x),
\end{split}
\end{equation}
where the regularization term $\Omega(x)$ is defined in~(\ref{prob.x}). Notice that the above subproblem~(\ref{eq:proxmapping}) has a closed-form solution, where for each $g\in \G$, we have
\begin{equation}\label{eq:proximal_x_kp1}
[x_{k+1}]_g=\max\left\{0,1-\alpha_k\lambda /\norm{[\widehat{x}_{k+1}]_g}\right\}\cdot [\widehat{x}_{k+1}]_g.
\end{equation}
In~\algacro{}, the Initialization Stage proceeds Prox-SG Step $N_\P$ times as a localization mechanism to seek an estimation which is close enough to a solution of problem~(\ref{prob.x}), where $N_\P:=\min\{k:k\in\mathbb{Z}^+, \norm{\bm{x}_k-\bm{x}^*}\leq R/2\}$ associated with a positive constant $R$ related to the optima, see~\eqref{def:R} in Appendix~\ref{appendix:convergence_analysis}. In practice, although the close-enough requirement is perhaps hard to be verified, we empirically suggest to keep running the Prox-SG Step until observing some stage-switch signal by testing on the stationarity of objective values, norm of (sub)gradient or validation accuracy similarly to~\citep{zhang2020statistical}. However, the Initialization Stage alone is \textit{insufficient} to exploit the group sparsity structure, \ie, the computed solution estimate is typically dense, due to the randomness and the moderate truncation mechanism of proximal operator constrained in its projection region, \ie, the trial iterate $[\widehat{x}_{k+1}]_g$ is projected to zero only if it falls into an $\ell_2$-ball centered at the origin with radius $\alpha_k\lambda $ by~(\ref{eq:proximal_x_kp1}). Our remedy is to incorporate it with the following Half-Space Step, which exhibits an effective sparsity promotion mechanism while still remains the convergent property. 

\begin{algorithm}[h]
	\caption{\halfspacestep}
	\label{alg:main.x.halfspacestep}
	\begin{algorithmic}[1]
		\State \textbf{Input:} Current iterate $x_k$, step size $ \alpha_k$, and $ \epsilon $.
		\State Compute the stochastic gradient of {$\Psi$} on $ \mathcal{I}^{\neq 0}(x_k)  $ by mini-batch $ \mathcal{B}_k$
		\begin{align}
		& [\Grad\Psi_{\mathcal{B}_k}(x_k)]_{\mathcal{I}^{\neq 0}(x_k) }\gets \frac{1}{|{\mathcal{B}_k}|}\sum_{i\in \mathcal{B}_k}[\Grad\Psi_i(x_k)]_{\mathcal{I}^{\neq 0}(x_k) }
		\end{align}\label{line:g_k_estimate_half_space_step}
		\State Compute
		$[\tilde{x}_{k+1}]_{\mathcal{I}^{\neq 0}(x_k)}\leftarrow [x_{k}-\alpha_k\Grad \Psi_{\mathcal{B}_k}(x_k)]_{\mathcal{I}^{\neq 0}(x_k) }$ and $[\tilde{x}_{k+1}]_{\mathcal{I}^{ 0}(x_k)}\leftarrow 0$.\label{line:half_space_trial_iterate}
		\For{each group $ g $ in $ \I^{\neq 0}(x_{k})  $}\label{line:half_space_project_start}
		\If{$ [\tilde{x}_{k+1}]_{g}^\top [x_{k}]_g <\epsilon\norm{[x_k]_g}^2$} \label{line:half_space_set_new_iterate_start}
		\State $ [\tilde{x}_{k+1}]_{g} \gets 0 $. \label{line:half_space_project}
		\EndIf
		\EndFor
		\State \textbf{Return}\ $x_{k+1}\gets \tilde{x}_{k+1}$.\label{line:half_space_set_new_iterate_end}
	\end{algorithmic}
\end{algorithm}

\paragraph{Group-Sparsity Stage.} 
The Group-Sparsity Stage is designed to effectively determine the groups of zero variables and capitalize convergence characteristic, which is in sharp contrast to other heuristic aggressive weight pruning methods but typically lacking theoretical guarantee~\citep{li2016pruning,luo2017thinet}. The underlying intuition of its atomic~\halfspacestep{}~(Algorithm~\ref{alg:main.x.halfspacestep}) is to project $[x_k]_g$ to zero only if $-[x_k]_g$ serves as a descent 
step to $\Psi(x_k)$, \ie, $-[x_k]_g^\top[\Grad \Psi(x_k))]_g<0$, hence updating $[x_{k+1}]_g\gets[x_k]_g-[x_k]_g=0$ still results in some progress to the optimality. Before introducing that, we first define the following index sets for any $ x\in\mathbb{R}^n $:
\begin{equation}\label{def:I_set}
\mathcal{I}^0(x) := \{g: g\in\mathcal{G}, [x]_g=0\}\ \text{and}\ \mathcal{I}^{\neq 0}(x) :=\{g: g\in\mathcal{G}, [x]_g\neq 0\},
\end{equation}
where $\I^{0}(x)$ represents the indices of groups of zero variables at $x$, and $\I^{\neq 0}(x)$ indexes the  groups of nonzero variables at $x$. To proceed, we further define an artificial set that $x$ lies in: 
\begin{equation}\label{def:polytope}
\S(x)\coloneqq \left\{z\in\mathbb{R}^n : [z]_g=0 \ \text{if}\ g\in\I^0{(x)},\text{and}\ [z]_g^T[x]_g\geq \epsilon \norm{[x]_g}^2\ \text{if}\ g\in\I^{\neq 0}(x)\right\} \bigcup \{0\},
\end{equation}
which consists of half-spaces and the origin. Here the parameter $\epsilon > 0$ controls the grey region presented in Figure~\ref{figure:project_region}, and the exact way to set $\epsilon$ will be discussed in Section~\ref{sec:exp} and Appendix. Hence, $x$ inhabits $\S(x_k)$, \ie, $x\in\S(x_k)$, only if: \textit{(i)} $[x]_g$ lies in the upper half-space for all $g\in\mathcal{I}^{\neq 0}(x_k)$ for some prescribed $\epsilon\in[0,1)$ as shown in Figure~\ref{figure:half_space_projection}; and \textit{(ii)} $[x]_g$ equals to zero for all $g\in\mathcal{I}^0(x_k)$. 
\begin{figure}[t]
	\centering
	\begin{subfigure}{0.48\textwidth}
		\includegraphics[scale=0.56]{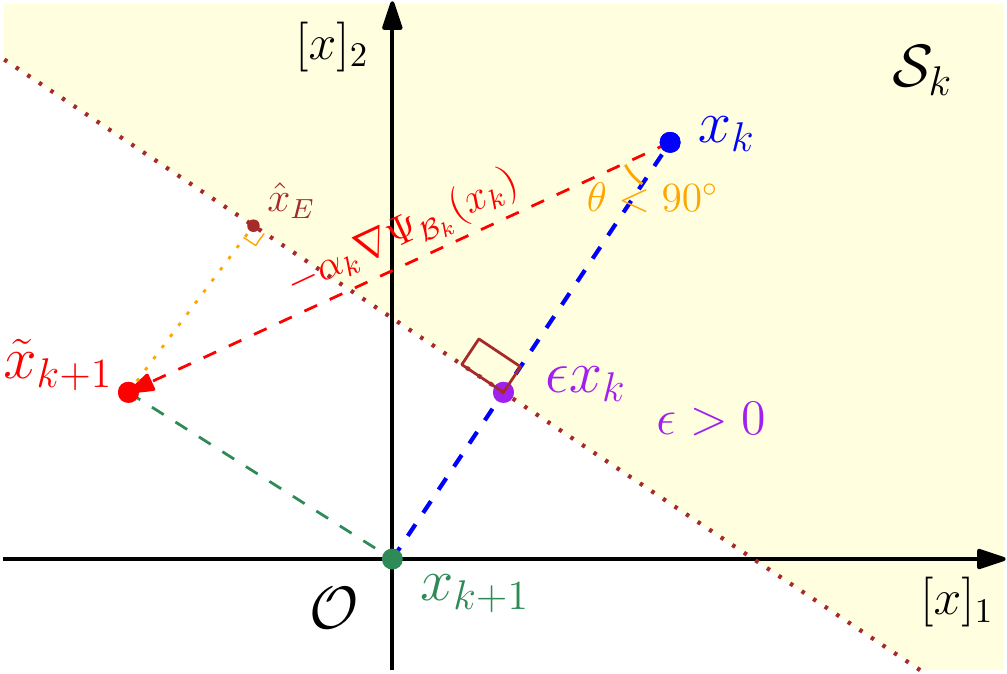}
		\caption{Half-Space Projection}
		\label{figure:half_space_projection}
	\end{subfigure}
	\begin{subfigure}{0.48\textwidth}
		\includegraphics[scale=0.6]{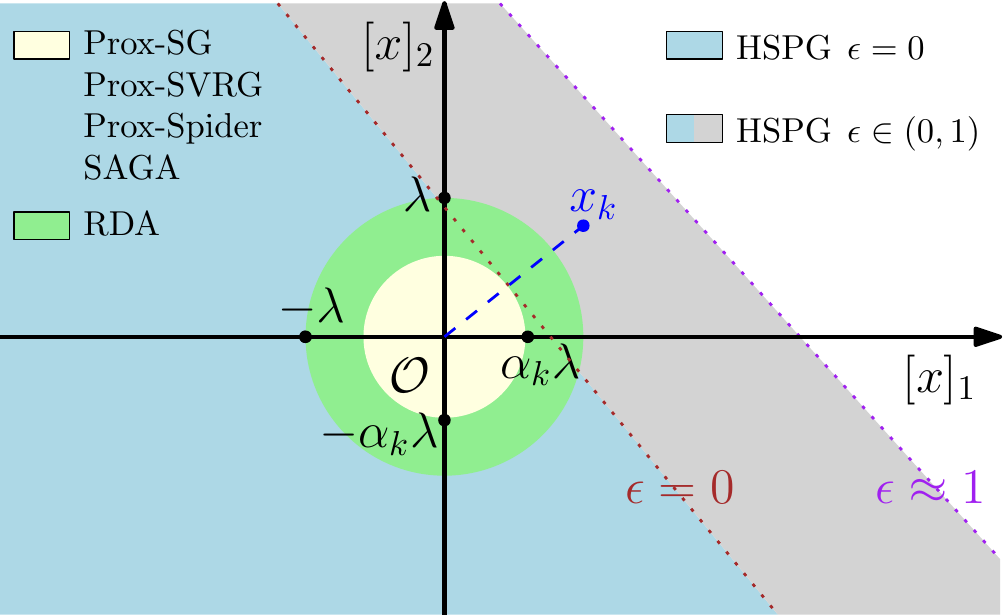}
		\caption{Projection Region}
		\label{figure:project_region}
	\end{subfigure}
	\caption{Illustration of Half-Space Step with projection in \eqref{def:proj}, where  $\mathcal{G}=\{\{1,2\}\}$.}
	\label{figure:proj_euclidean}
	\vspace{-0.3cm}
\end{figure}
The fundamental assumption for~Half-Space Step to success is that:  the~Initialization Stage has produced a (possibly \textit{non-sparse}) solution estimate $x_k$ nearby a group sparse solution $x^*$ of problem~(\ref{prob.x}), \ie, the optimal distance $\norm{x_k-x^*}$ is sufficiently small. As seen in Appendix, it further indicates that the group sparse optimal solution $x^*$ inhabits $\S_k:=\S(x_k)$, which implies that $\S_k$ has already covered the group-support of $x^*$, \ie, $\I^{\neq 0}(x^*)\subseteq \I^{\neq 0}(x_k)$. Our goal now becomes minimizing $\Psi(x)$ over $\S_k$ to identify the remaining groups of zero variables, \ie, $\I^0(x^*)/\I^{0}(x_k)$, which is formulated as the following smooth optimization problem: 
\begin{equation}\label{prob.half_space_sub_problem}
x_{k+1}=\argmin_{x\in\S_k}\ \Psi(x)=f(x)+\lambda\Omega(x).
\end{equation}
By the definition of $\S_k$, $[x]_{\mathcal{I}^0(x_k)}\equiv 0$ are constrained as fixed during Algorithm~\ref{alg:main.x.halfspacestep} proceeding, and only the entries in $\I^{\neq 0}(x_k)$ are allowed to move. Hence $\Psi(x)$ is smooth on $\S_k$, and~(\ref{prob.half_space_sub_problem}) is a reduced space optimization problem.
A standard way to solve problem~(\ref{prob.half_space_sub_problem}) would be the stochastic gradient descent equipped with Euclidean projection~\citep{nocedal2006numerical}. However, such a projected method rarely produces zero (group) variables as the dense $\hat{x}_{E}$ illustrated in Figure~\ref{figure:half_space_projection}. To address it, we introduce a novel projection operator to effectively conduct group projection as follows.

As stated in Algorithm~\ref{alg:main.x.halfspacestep}, we first approximate the gradient of $\Psi$ on the free variables in $\I^{\neq 0}(x_k)$ by $[\Grad \Psi_{\B_k}(x_k)]_{\I^{\neq 0}(x_k)}$ (line~\ref{line:g_k_estimate_half_space_step}, Algorithm~\ref{alg:main.x.halfspacestep}), then employ SGD to compute a trial point $\widetilde{x}_{k+1}$ (line~\ref{line:half_space_trial_iterate}, Algorithm~\ref{alg:main.x.halfspacestep}) which is passed into a new projection operator $\proj_{\S_k}(\cdot)$ defined as
\begin{equation}\label{def:proj}
\left[\proj_{\S_k}(z)\right]_g\coloneqq\bigg\{
\begin{array}{ll}
[z]_g & \text{if}\ [z]_g^T[x_k]_g\geq \epsilon\norm{[x_k]_g}^2,\\ 
0 & \text{otherwise}.
\end{array}
\end{equation}
The above projector of form~(\ref{def:proj}) is not the standard Euclidean projection operator in most cases\footnote{ Unless $\Omega(x)$ is $\norm{x}_1$ where each $g\in\G$ is singleton, then $\S_k$ becomes an orthant face~\citep{chen2020orthant}.}, but still satisfies the following two advantages: \textit{(i)} the actual search direction $d_k:=(\proj_{\S_k}(\tilde{x}_{k+1})-x_k)/\alpha_k$ performs as a descent direction to $\Psi_{\B_k}(x_k):=f_{\B_k}(x_k)+\lambda\Omega(x_k)$, \ie, $[d_k]_g^\top[\Grad \Psi_{\B_k}(x_k))]_g<0$ as $\theta<90^{\circ}$ in Figure~\ref{figure:half_space_projection}, then the progress to the optimum is made via the sufficient decrease property as drawn in Lemma~\ref{lemma:sufficient_decrease_half_space}; and \textit{(ii)} effectively project groups of variables to zero simultaneously if the inner product of corresponding entries is sufficiently small. In contrast, the Euclidean projection operator is far away effective to promote group sparsity, as the Euclidean projected point $\hat{x}_E\neq 0$ versus $x_{k+1}=\proj_{\S_k}(\tilde{x}_{k+1})=0$ shown in Figure~\ref{figure:half_space_projection}.
\begin{lemma}\label{lemma:sufficient_decrease_half_space}
	Algorithm~\ref{alg:main.x.halfspacestep} yields the next iterate $x_{k+1}$ as $\textit{Proj}_{\S_k}(x_k-\alpha_k\Grad \Psi_{\B_k}(x_k))$, then the search direction $d_k:=(x_{k+1}-x_k)/\alpha_k$ is a descent direction for $\Psi_{\B_k}(x_k)$, \ie, $d_k^\top\Grad \Psi_{\B_k}(x_k)<0$. Moreover, letting $L$ be the Lipschitz constant for $ \Grad \Psi_{\mathcal{B}_k} $ on the feasible domain, and $\hat{\mathcal{G}}_k:=\I^{\neq 0}(x_k)\bigcap \I^{0}(x_{k+1})$ and $\tilde{\mathcal{G}}_k:=\I^{\neq 0}(x_k)\bigcap \I^{\neq 0}(x_{k+1})$ be the sets of groups which projects or not onto zero, we have
	\begin{equation}
	\small
	\begin{split}
	\Psi_{\mathcal{B}_k}(x_{k+1})\leq & \Psi_{\mathcal{B}_k}(x_{k})-\left(\alpha_k-\frac{\alpha_k^2L}{2}\right)\sum_{g\in\tilde{\G}_k}\norm{[\Grad \Psi_{\mathcal{B}_k}(x_{k})]_g}^2-\left(\frac{1-\epsilon}{\alpha_k}-\frac{L}{2}\right)\sum_{g\in\hat{\G}_k}\norm{[x_{k}]_g}^2.
	\end{split}
	\end{equation}
	\vspace{-0.4cm}
\end{lemma}
We then intuitively illustrate the strength of~\algacro{} on group sparsity exploration. In fact, the half-space projection~(\ref{def:proj}) is a more effective sparsity promotion mechanism compared to the existing methods. Particularly, it benefits from a much larger projection region to map a reference point $\hat{x}_{k+1}:=x_k-\alpha_k\Grad f_{\mathcal{B}_k}(x_k)$ or its variants to zero. As the 2D case described in Figure~\ref{figure:project_region}, the projection regions of~Prox-SG, Prox-SVRG,~\proxspider{} and SAGA are $\ell_2$-balls with radius as $\alpha_k\lambda$. In stochastic learning, especially deep learning tasks, the step size $\alpha_k$ is usually selected around $10^{-3}$ to $10^{-4}$ or even smaller for convergence. Together with the common setting of $\lambda \ll 1$, their projection regions would vanish rapidly, resulting in the difficulties to produce group sparsity. As a sharp contrast, even though $\alpha_k\lambda$ is near zero, the projection region of~\algacro{} $\{x: x_k^Tx< (\alpha_k\lambda + \epsilon\norm{x_k})\norm{x_k}\}$ (seen in Appendix) is still an open half-space which contains those $\ell_2$ balls as well as~\rda's if $\epsilon$ is large enough. 
Moreover, the positive control parameter $\epsilon$ adjusts the level of aggressiveness of group sparsity promotion~(\ref{def:proj}),~\ie, the larger the more aggressive, and meanwhile maintains the progress to the optimality by Lemma~\ref{lemma:sufficient_decrease_half_space}. In practice, proper fine tuning $\epsilon$ is sometimes required to achieve both group sparsity enhancement and sufficient decrease on objective value as will see in Section~\ref{sec:exp}.

\textbf{Intuition of Two-Stage Method:} To end this section, we discuss the advantage of designing such two stage schema rather than an adaptive switch back and forth between the Prox-SG Step and~\halfspacestep{} based on some evaluation switching criteria, 
as many multi-step deterministic optimization algorithms~\citep{chen2017reduced}. In fact, we numerically observed that switching back to the Prox-SG Step consistently deteriorate the progress of group sparsity exploration by~\halfspacestep{} while without obvious gain on convergence. Such regression on group sparsity by the Prox-SG Step is less attractive in realistic applications, \eg, model compression, where people usually possess heavy models of high generalization accuracy ahead and want to filter out the redundancy effectively. Therefore, in term of the ease of application, we end at organizing Prox-SG Step and~\halfspacestep{} as such a two-stage schema, controlled by a switching hypermeter $N_\P$. In theory, we require $N_\P$ sufficiently large to let the initial iterate of~\halfspacestep{} be close enough to the local minimizer as shown in Section~\ref{sec:convergence}. In practice,~\algacro{} is sensitive to the choice of $N_\P$ at early iterations, \ie, switching to~\halfspacestep{} too early may result accuracy loss. But such sensitivity vanishes rapidly if switching to~\halfspacestep{} after some acceptable evaluation switching criteria. 

\section{Convergence Analysis}\label{sec:convergence}

In this section, we give the convergence guarantee of our \algacro{}.
Towards that end, we make the following widely used  assumption in optimization literature~\citep{xiao2014proximal,yang2019stochastic} and active set identification analysis of regularization problem~\citep{nutini2019active,chen2018farsa}.

\begin{assumption}\label{assumption}
	Each $f_i:\mathbb{R}^n\to \mathbb{R}$, for $i=1,2,\cdots, N$, is differentiable and bounded below. Their gradients $\Grad f_{i}(\bm{x})$ are Lipschitz continuous, and let $L$ be the shared Lipschitz constant. 
\end{assumption}
\begin{assumption}
	The least and the largest $\ell_2$-norm of non-zero groups in $\bm{x}^*$ are lower and upper bounded by some constants, \ie, $0<2\delta_1 :=\min_{g\in \mathcal{I}^{\neq 0}(\bm{x}^*)}\norm{[\bm{x}^*]_g}$ and $\ 0<2\delta_2 :=\max_{g\in \mathcal{I}^{\neq 0}(\bm{x}^*)}\norm{[\bm{x}^*]_g}$. Moreover, we request a common strict complementarity on any $\B$, \ie, $0<2\delta_3:=\min_{g\in\mathcal{I}^0(\bm{x}^*)}\left(\lambda - \norm{[\Grad f_{\B}(\bm{x}^*)]_g}\right)$ for regularization optimization.
\end{assumption}
\textbf{Notations:} 
Let $\bm{x}^*$ be a local minimizer of problem (\ref{prob.x}) with group sparsity property, $\Psi^*$ be the local minimum value corresponding to $\bm{x}^*$, and $\{\bm{x}_k\}_{k=0}^\infty$ be the iterates generated from Algorithm~\ref{alg:main.x.outline}. Denote the gradient mapping of $\Psi(\bm{x})$ and its estimator on mini-batch $\mathcal{B}$ as $\bm{\bm{\xi}}_{\eta}(\bm{x}):=\frac{1}{\eta}\left(\bm{x}-\prox_{\eta\lambda\Omega(\cdot)}(\bm{x}-\eta \Grad f(\bm{x}))\right)$ and $\bm{\bm{\xi}}_{\eta,\mathcal{B}}(\bm{x}):=\frac{1}{\eta}\left( \bm{x}-\prox_{\eta\lambda\Omega(\cdot)}(\bm{x}-\eta \Grad f_\mathcal{B}(\bm{x}))\right)$ respectively. We say $\tilde{\bm{x}}$ a stationary point of $\Psi(\bm{x})$ if $\bm{\bm{\xi}}_{\eta}(\tilde{\bm{x}})=0$. 
To be simple, let $\widetilde{\mathcal{X}}$ be a neighbor of $\bm{x}^*$ as $\widetilde{\mathcal{X}}:=\{\bm{x}: \norm{\bm{x}-\bm{x}^*}\leq R\}$ with $R$ as a positive constant related to $\delta_1, \delta_2$ and $\epsilon$ (see~\eqref{def:R} in Appendix~\ref{appendix:convergence_analysis}), and $M$ be the supremum of $\norm{\partial \Psi(\bm{x})}$ on the compact set $\widetilde{\mathcal{X}}$. 

\textbf{Remark:} Assumption~\ref{assumption} implies that $\Grad f_{\B}(\bm{x})$ measured on mini-batch $\B$ is Lipschitz continuous on $\mathbb{R}^n$ with the same Lipschitz constant $L$, while $\Grad \Psi_{\B}(\bm{x})$ is not as shown in Appendix. However, the Lipschitz continuity of $\Grad \Psi_{\B}(\bm{x})$ still holds on $\mathcal{X}=\{\bm{x}: \norm{[\bm{x}]_g}\geq \delta_1\ \text{for each}\ g\in\mathcal{G}\}$ by excluding a $\ell_2$-ball centered at the origin with radius $\delta_1$ from $\mathbb{R}^n$. For simplicity, let $\Grad \Psi_\B(\bm{x})$ share the same Lipschitz constant $L$ on $\mathcal{X}$ with $\Grad f_\B(\bm{x})$, since we can always select the bigger value as their shared Lipschitz constant. Now, we state the first main theorem of~\algacro{}.
\begin{theorem}\label{thm:convergence}
	Suppose $f$ is convex on $\widetilde{\mathcal{X}}$, $\epsilon\in\left[0,\min\left\{\frac{\delta_1^2}{\delta_2}, \frac{2\delta_1-R}{2\delta_2+R}\right\}\right)$, $\norm{\bm{x}_{K}-\bm{x}^*}\leq\frac{R}{2}$ for $K\geq N_\P$. Set $k:=K+t$, $(t\in\mathbb{Z}^+)$. Then for any $\tau\in(0,1)$, there exist step size $\alpha_k=\O(\frac{1}{\sqrt{N}t})\in\left(0,\min\left\{\frac{2(1-\epsilon)}{L}, \frac{1}{L},\frac{2\delta_1-R-\epsilon(2\delta_2+R)}{M}\right\}\right)$, and mini-batch size $|\B_k|=\O(t)\leq N-\frac{N}{2M}$, such that $\{\bm{x}_k\}$ converges to some stationary point in expectation with probability at least $1-\tau$, \ie, $\mathbb{P}(\lim_{k\rightarrow \infty} \mathbb{E}\left[ \norm{\bm{\bm{\xi}}_{\alpha_k,\mathcal{B}_k}(\bm{x}_k)}\right]=0)\geq 1-\tau$.
\end{theorem}


\textbf{Remark:} Theorem~\ref{thm:convergence} only requires local convexity of $f$ on a neighborhood $\widetilde{\mathcal{X}}$ of $\bm{x}^*$ while itself can be non-convex in general. This local convexity assumption appears in many non-convex analysis, such as: tensor decomposition~\citep{ge2015escaping} and shallow neural networks~\citep{zhong2017recovery}. Theorem~\ref{thm:convergence} implies that if the $K$th iterate locates close enough to $\bm{x}^*$, the step size $\alpha_k$ and mini-batch size $|\B_k|$ is set as above, (it further indicates $\bm{x}^*$ inhabits the $\{\S_k\}_{k\geq K}$ of all subsequent iterates updated by~\halfspacestep{} with high probability in Appendix), then the~\halfspacestep{} in Algorithm~\ref{alg:main.x.halfspacestep} guarantees the convergence to the stationary point. The $\O(t)$ mini-batch size is commonly used in 
the analysis of stochastic algorithms, \eg, Adam and Yogi~\citep{zaheer2018adaptive}. Later based on numerical results in Section~\ref{sec:exp}, we observe that a much weaker increasing or even constant mini-batch size is sufficient. In fact, experiments show that practically, a reasonably large mini-batch size can work well if the variance is not large. Although the assumption $\norm{\bm{x}_{K}-\bm{x}^*}<R/2$ is hard to be verified in practice, setting $N_\mathcal{P}$  large enough usually performs quite well.

We then reveal the sparsity identification guarantee of~\algacro{} as stated in Theorem~\ref{thm:sparsity_recovery_rate_hbproxsg}.

\begin{theorem}\label{thm:sparsity_recovery_rate_hbproxsg}
	If $k\geq N_\P$ and $\norm{\bm{x}_k-\bm{x}^*}\leq \frac{2\alpha_k\delta_3}{1-\epsilon+\alpha_kL}$, then~\algacro{} yields $\I^0(\bm{x}^*)\subseteq \I^0(\bm{x}_{k+1})$. 
\end{theorem}

\textbf{Remark:} Theorem~\ref{thm:sparsity_recovery_rate_hbproxsg} shows that when $\bm{x}_k$ is in the $\ell_2$-ball centered at $\bm{x}^*$ with radius $\frac{2\alpha_k\delta_3}{1-\epsilon+\alpha_kL}$,~\algacro{} identifies the optimal sparsity pattern, \ie, $\I^0(\bm{x}^*)\subseteq \I^0(\bm{x}_{k+1})$. In contrast, to identify the sparsity pattern,~\proxsg{} requires the iterates to fall into the $\ell_2$-ball centered at $\bm{x}^*$ with radius $\alpha_k \delta_3$~\citep{nutini2019active}. Since $\alpha_k\leq 1/L$ and $\epsilon\in [0,1)$, then $\frac{2\alpha_k\delta_3}{1-\epsilon+\alpha_kL}\geq\alpha_k\delta_3$ implies that the $\ell_2$-ball of~\algacro{} contains the $\ell_2$-ball of~\proxsg{}, \ie,~\algacro{} has a stronger performance in sparsity pattern identification.  Therefore, Theorem~\ref{thm:sparsity_recovery_rate_hbproxsg} reveals a better sparsity identification property of~\algacro{} than~\proxsg{}, and no similar results exist for other methods to our knowledge.

\noindent
\textbf{The Initialization Stage Selection:} To satisfy the pre-requirement of convergence of Half-Space Step as Theorem~\ref{thm:convergence}, \ie, initial iterate close enough to $\bm{x}^*$,
there exists several proper candidates \eg,~\proxsg{},~\proxsvrg{} and~\saga{}  to form as the Initialization Stage. 
Considering the tradeoff between computational efficiency and theoretical convergence, our default setting is to select~\proxsg{}. Although ~\proxsvrg/SAGA may have better theoretical convergence property than~\proxsg{}, they require higher time and space complexity to compute or estimate full gradient on a huge mini-batch or store previous gradient, which may be prohibitive for large-scale training especially when the memory is often limited. Besides, it is well noticed that SVRG does not work as desired on the popular non-convex deep learning applications~\citep{defazio2019ineffectiveness, chen2020orthant}. In contrast, Prox-SG is efficient and can also achieves the good initialization assumption in Theorem~\ref{thm:convergence}, \ie, $\norm{\bm{x}_{N_\P}-\bm{x}^*}\leq R/2$, in the manner of high probability via performing sufficiently many times, as revealed in Appendix~\ref{appendx:upper_bound_n_p} by leveraging related literature~\citep{rosasco2019convergence} associated with an additional strongly convex assumption. However, one should notice that Prox-SG does not guarantee any group sparsity property of $\bm{x}_{N_\mathcal{P}}$ due to the limited projection region and randomness.

\textbf{Remark}: We emphasize that this paper focuses on improving the group sparsity identification, which is rarely explored and also a key indicator of success for structured sparsity regularization problem. Meanwhile, we would like to point out improving the convergence rate has been very well explored in a series of literatures~\citep{reddi2016proximal,li2018simple}, but out of our main consideration.

\section{Numerical Experiments}
\label{sec:exp}

In this section, we present results of several benchmark numerical experiments in deep neural networks to illustrate the superiority of~\algacro{} than other related algorithms on group sparsity exploration and the comparable convergence. Besides, two extensible convex experiments are conducted in Appendix to empirically demonstrate the validness and superiority of the group sparsity identification by~\algacro{}.

\vspace{-0.1in}
\paragraph{Image Classification:} 

We now consider the popular Deep Convolutional Neural Networks (DCNNs) for image classification tasks. 
Specifically, we select several popular and benchmark DCNN architectures, \ie,  \vgg{}~\citep{simonyan2014very}, \resnet{}~\citep{he2016deep} and~\mobilenet{}~\citep{howard2017mobilenets} on two benchmark datasets CIFAR10~\citep{Krizhevsky09} and~\fashionmnist{}~\citep{xiao2017online}.
We conduct all experiments 
for 300 epochs with a mini-batch size of 128 and $\lambda$ as $10^{-3}$, since it returns competitive testing accuracy to the models trained without regularization, (see more in Appendix~\ref{appendix:nonconvex_exp}). 
The step size $\alpha_k$ is initialized as $0.1$, and decayed by a factor 0.1 periodically. 
We set each filter in the convolution layers as a group variable.

In these experiments, we proceed a test on the objective value stationarity similarly to~\citep[Section 2.1]{zhang2020statistical} and switch to Half-Space Step roughly on 150 epochs with $N_\mathcal{P}$ as $150N/|\mathcal{B}|$. The control parameter $\epsilon$ in the half-space projection~\eqref{def:proj} controls the aggressiveness level of group sparsity promotion, which is first set as 0, then fined tuned to be around $0.02$ to favor the sparsity level whereas does not hurt the target objective $\Psi$; the detailed procedure is in Appendix. We exclude~\rda{} because of no acceptable experimental results attained during our tests with the step size parameter $\gamma$ setting throughout all powers of 10 from $10^{-3}$ to $10^3$, and skip Prox-Spider and SAGA since Prox-SVRG has been a superb representative
to the proximal incremental gradient methods. 

Table~\ref{table:nonconvex} demonstrates the effectiveness and superiority of~\algacro{}, where we mark the best values as bold, and the group sparsity ratio is defined as the percentage of zero groups. In particular, \textit{(i)} \algacro{} computes remarkably higher group sparsity than other methods on all tests under both $\epsilon=0$ and fine tuned $\epsilon$, of which the solutions are typically multiple times sparser in the manner of group than those of~\proxsg{}, while~\proxsvrg{} performs not comparably since the variance reduction techniques may not work as desired for deep learning applications~\citep{defazio2019ineffectiveness}; \textit{(ii)} \algacro{} performs competitively with respect to the final objective values $\Psi$ and $f$ (see $f$ in Appendix). In addition, all the methods reach a comparable generalization performance on unseen test data. On the other hand, sparse regularization methods may yield solutions with entries that are not exactly zero but are very small. Sometimes all entries below certain threshold ($\mathcal{T}$) are set to zero~\citep{jenatton2010structured,el2018combinatorial}. However, such simple truncation mechanism is heuristic-rule based, hence may hurt convergence and accuracy. To illustrate this, we set the groups of the solutions of~\proxsg{} and~\proxsvrg{} to zero if the magnitudes of the group variables are less than some $\mathcal{T}$, and denote the corresponding solutions as \proxsg{}* and~\proxsvrg{}*.
As shown in Figure~\ref{figure:simple_truncation}\textcolor{red}{(i)}, under the $\mathcal{T}$ with no accuracy regression,~\proxsg{}* and~\proxsvrg{}* reach higher group sparsity ratio as 60\% and 32\% compared to Table~\ref{table:nonconvex}, but still significantly lower than the 70\% of HSPG under $\epsilon=0.05$ without simple truncation. Under the $\mathcal{T}$ to reach the same group sparsity ratio as HSPG, the testing accuracy of~\proxsg{}* and~\proxsvrg{}* regresses drastically to 28\% and 17\% in Figure~\ref{figure:simple_truncation}\textcolor{red}{(ii)} respectively. Remark here that although further refitting the models from \proxsg{}* and~\proxsvrg{}* on active (non-zero) groups of weights may recover the accuracy regression, it requires additional engineering efforts and training cost, which is less attractive and convenient than~\algacro{} (with no need to refit).

\begin{table}[t]
	\centering
	\caption{Final $\Psi$/group sparsity ratio/testing accuracy for tested algorithms on non-convex problems.}
	\label{table:nonconvex}
	\resizebox{\textwidth}{!}{
		\begin{tabular}{ 
				cccccc}
			\Xhline{3\arrayrulewidth}
			\multirow{2}{*}{Backbone} & \multirow{2}{*}{Dataset} &\multirow{2}{*}{\proxsg{}} &  \multirow{2}{*}{\proxsvrg{}}  & \multicolumn{2}{c}{\algacro{}}\\
			\cline{5-6}
			& & &   & $\epsilon$ as $0$ & fine tuned $\epsilon$ \\
			\hline
			\multirow{2}{*}{\vgg{}} & \cifar{} & \textbf{0.59}\ /\ 53.95\%\ /\ 90.57\%  & 0.82\ /\ 14.73\%\ /\ 89.42\% &  \textbf{0.59}\ /\ 74.60\%\ /\ \textbf{91.10\%} & \textbf{0.59}\ /\ \textbf{75.61\%} \ /\ 90.92\%  \\
			& \fashionmnist{} & 0.54\ /\ 15.63\%\ /\ \textbf{92.99\%}  & 2.66\ /\ 0.45\%\ /\ 92.69\%  & 0.54\ /\ 22.18\%\ /\ 92.98\% &  \textbf{0.53}\ /\ \textbf{60.77}\%\ /\ 92.87\% \\\hdashline
			\multirow{2}{*}{\resnet{}} & \cifar{} & \textbf{0.31}\ /\ 19.50\%\ /\ 94.09\%  & 0.36\ /\ 2.79\%\ /\ 94.17\%  &  \textbf{0.31}\ /\ 41.58\%\ /\ 94.39\% & \textbf{0.31}\ /\ \textbf{62.97\%}\ /\ \textbf{94.53\%} \\
			& \fashionmnist{} & 0.14\ /\ 0.00\%\ /\ 94.82\% & 0.19\ /\ 0.00\%\ /\ 94.64\% & \textbf{0.13}\ /\ 6.60\%\ /\ \textbf{94.93\%} & \textbf{0.13}\ /\ \textbf{63.93\%}\ /\ 94.86\%\\
			\hdashline
			\multirow{2}{*}{\mobilenet{}} & \cifar{} & \textbf{0.40}\ /\ 57.81\% \ /\ 91.60\% & 0.65\ /\ 32.22\% \ /\ 90.08\% & \textbf{0.40}\ /\ 65.04\% \ /\ \textbf{91.86\%} & 0.41\ /\ \textbf{71.66\%} \ /\ 91.54\% \\
			& \fashionmnist{} & \textbf{0.22}\ /\ 65.80\% \ /\ 94.36\%  & 0.48\ /\ 38.76\%  \ /\ 93.95\% & 0.23\ /\ 74.52\% \ /\ 94.43\% & 0.24\ /\ \textbf{83.71\%}\ /\ \textbf{94.44\%} \\
			\Xhline{3\arrayrulewidth} 
	\end{tabular}}
	
	\vspace{-0.1in}
\end{table} 

\vspace{-0.1in}
\begin{figure}[t!]
	\centering
	\begin{subfigure}[t]{0.24\textwidth}
		\includegraphics[width=\linewidth]{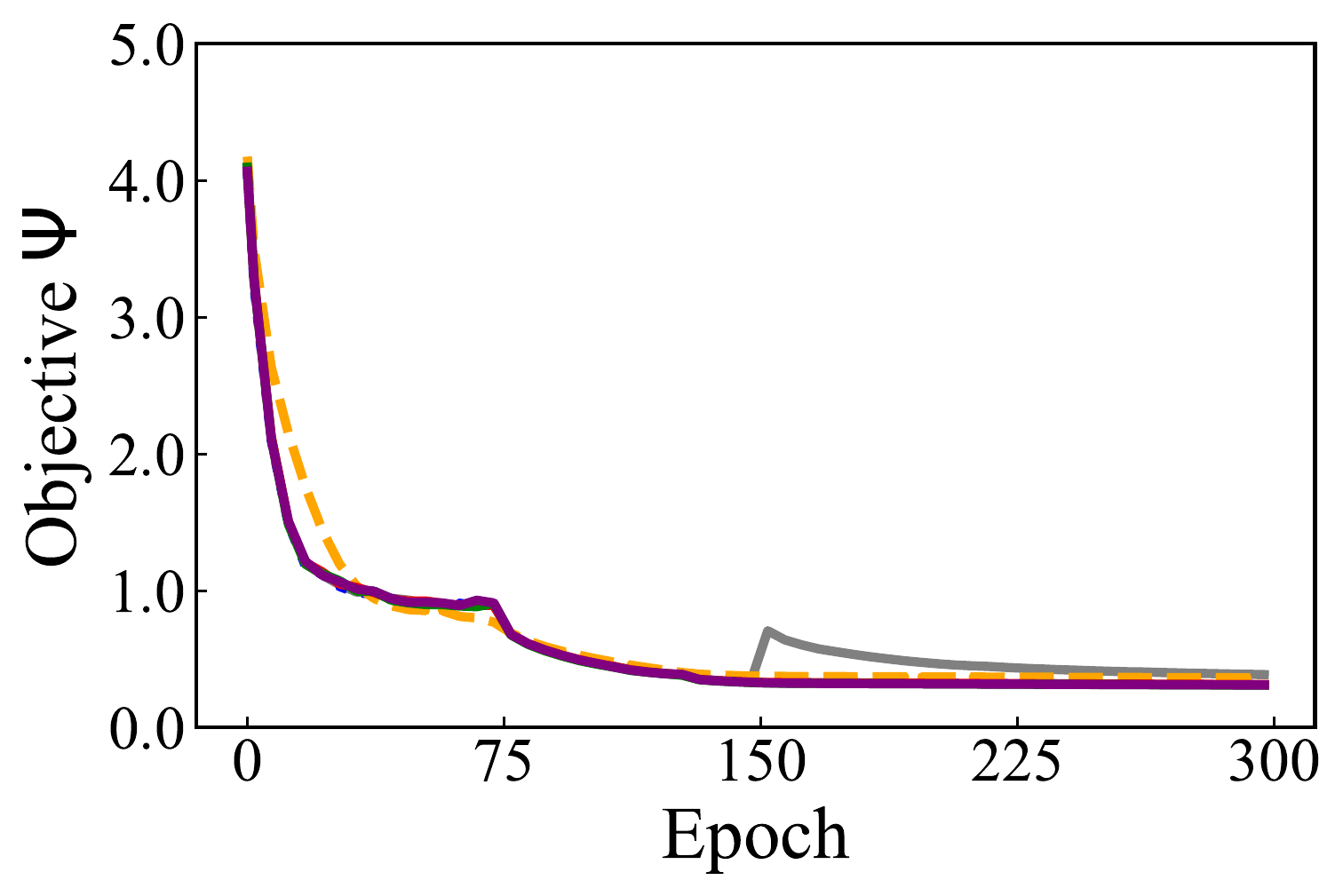}
		\vspace{-0.2in}
		\caption{\scriptsize Objective $\Psi$}
		\vspace{-0.1in}
	\end{subfigure}
	\begin{subfigure}[t]{0.24\textwidth}
		\includegraphics[width=\linewidth]{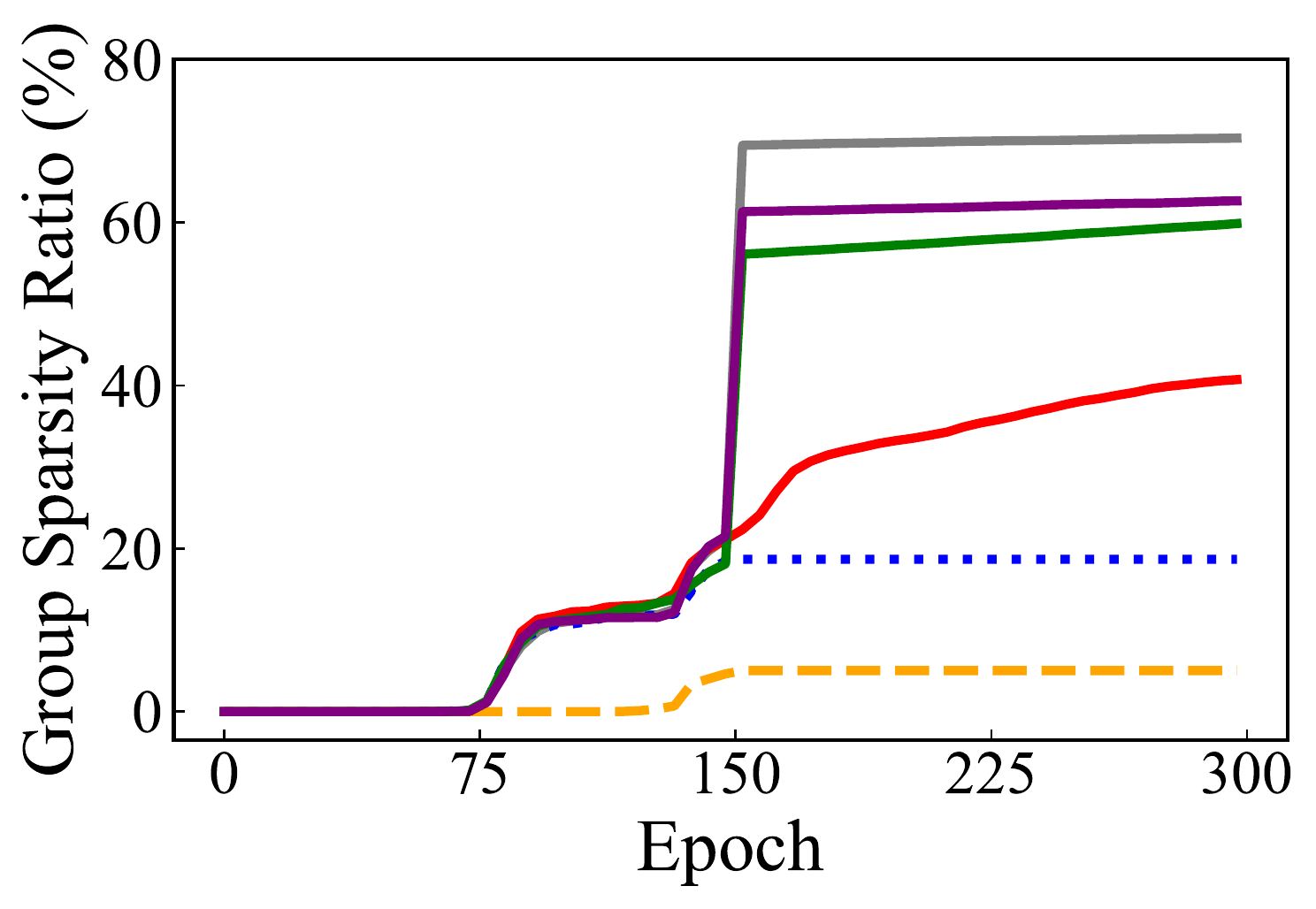}
		\vspace{-0.2in}
		\caption{\scriptsize  Group Sparsity Ratio}
		\label{figure:group_sparsity}
		\vspace{-0.1in}
	\end{subfigure}
	\begin{subfigure}[t]{0.24\textwidth}
		\includegraphics[width=\linewidth]{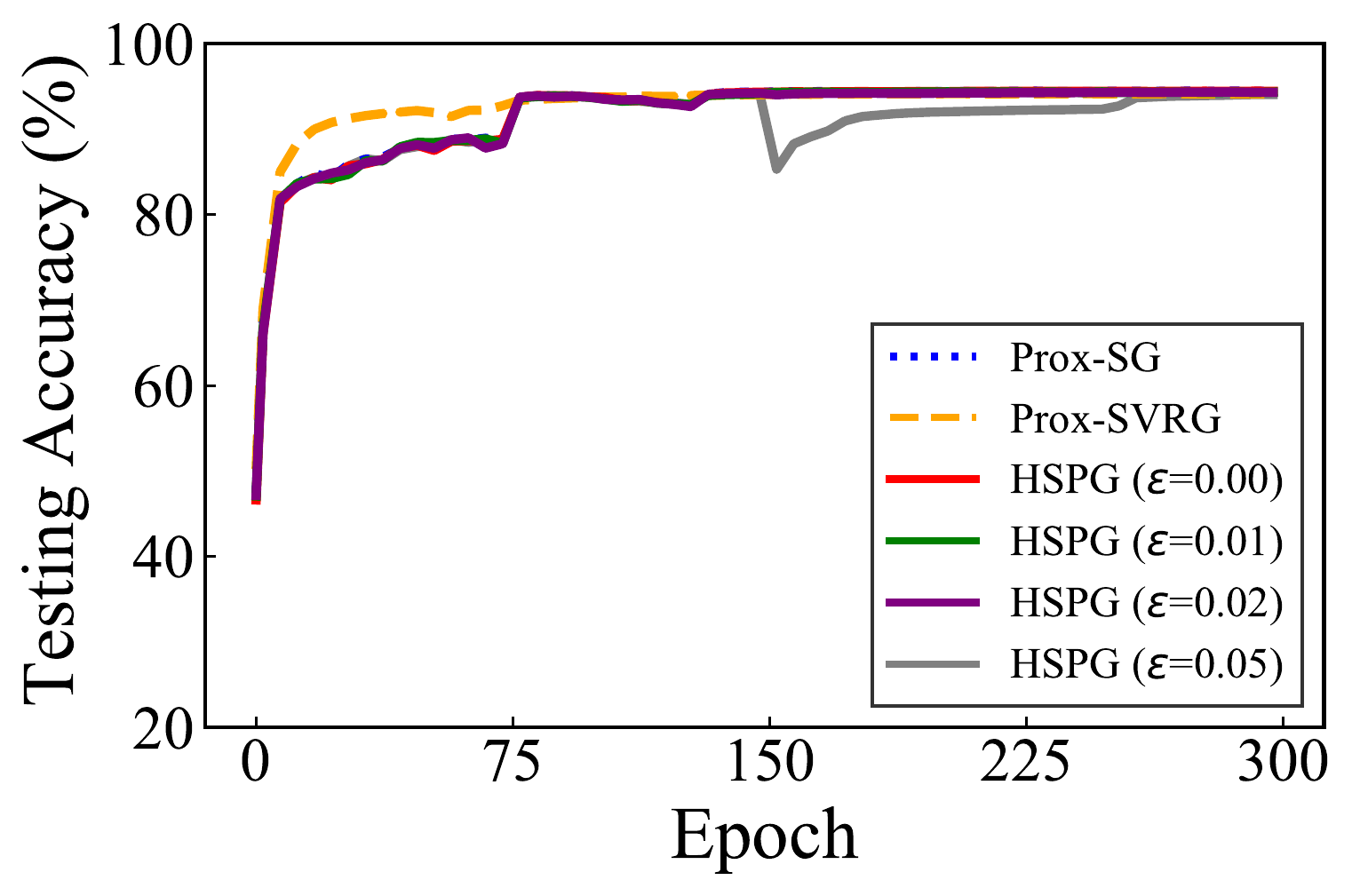}
		\vspace{-0.2in}
		\caption{\scriptsize  Testing Accuracy}
		\vspace{-0.1in}
	\end{subfigure}
	\begin{subfigure}[t]{0.24\textwidth}
		\includegraphics[width=\linewidth]{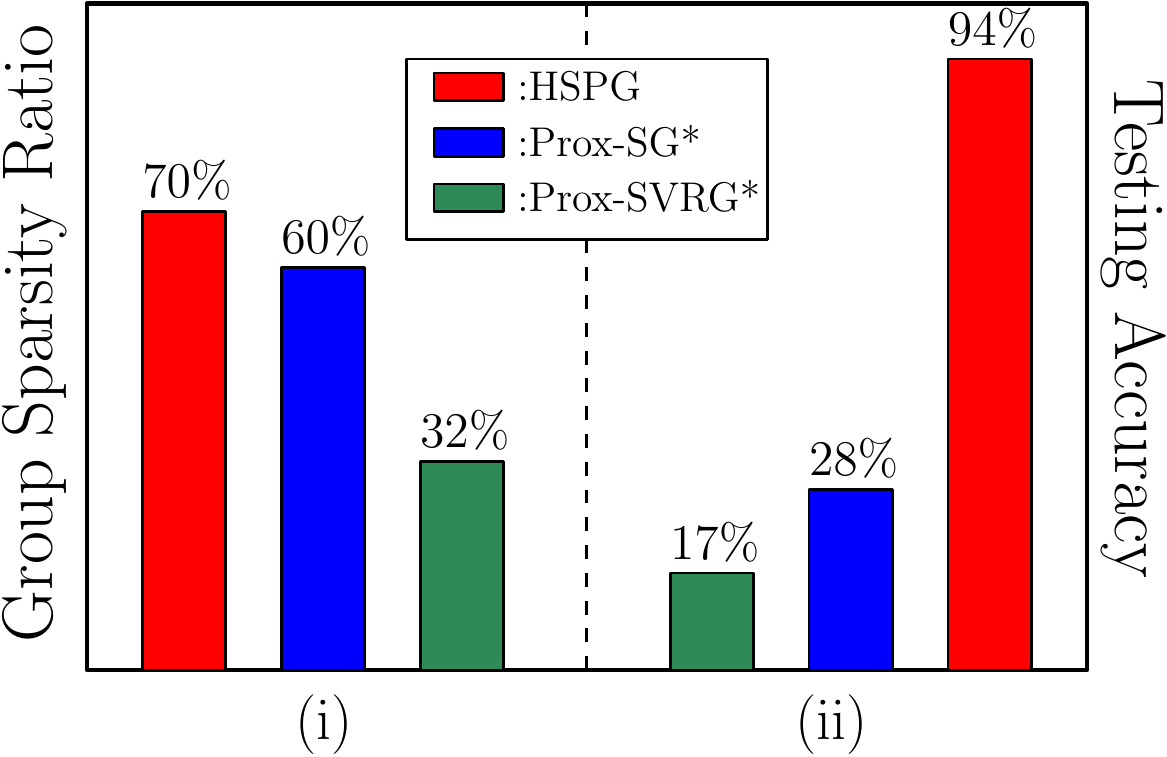}
		\vspace{-0.2in}
		\caption{{\scriptsize HSPG VS  Truncation}}
		\label{figure:simple_truncation}
		\vspace{-0.1in}
	\end{subfigure}
	\caption{\small On~\resnet{} with~\cifar{}, (a)-(c): Evolution of $\Psi$, group sparsity ratio and testing accuracy, (d):~\algacro{} versus~\proxsg{}* and~\proxsvrg{}* (\proxsg{} and \proxsvrg{} with simple truncation mechanism).}
	\vspace{-0.2in}
\end{figure}

\vspace{0.1in}

Finally, we investigate the group sparsity evolution under different $\epsilon$'s. As shown in Figure~\ref{figure:group_sparsity}, \algacro{} produces the highest group-sparse solutions compared with other methods. Notably, at the early $N_{\mathcal{P}}$ iterations, ~\algacro{} performs merely the same as~\proxsg{}. However, after switching to \halfspacestep{} at the 150th epoch, \algacro{} outperforms all the other methods dramatically, and larger $\epsilon$ results in higher sparsity level. It is a strong evidence that our half-space based technique is much more successful than the proximal mechanism and its variants in terms of the group sparsity identification. Besides, the evolutions of $\Psi$ and testing accuracy confirm the comparability on convergence among the tested algorithms. Particularly, the objective $\Psi$ generally monotonically decreases for small $\epsilon=0$ to $0.02$, and experiences a mild pulse after switch to~\halfspacestep{} for larger $\epsilon$, \eg, 0.05, which matches Lemma~\ref{lemma:sufficient_decrease_half_space}. As a result, with the similar generalization accuracy, \algacro{} allows dropping entire hidden units of networks, which may further achieve automatic dimension reduction and construct smaller model architectures for efficient inference.  

\section{Conclusions and Future Work}

\vspace{-0.05in}
We proposed a new~\algname{} (\algacro) method for disjoint group-sparsity induced regularized problem, which can be applied to various structured sparsity stochastic learning problem.
\algacro{} makes use of proximal stochastic gradient method to seek a near-optimal solution estimate, followed by a novel half-space group projection to effectively exploit the group sparsity structure. In theory, we provided the convergence guarantee, and showed its better sparsity identification performance. Experiments on both convex and non-convex problems demonstrated that~\algacro{} usually achieves solutions with competitive objective values and significantly higher group sparsity compared with state-of-the-arts stochastic solvers. Further study is needed to investigate the proper leverage of group sparsity into diverse deep learning applications, \eg, help people design and understand optimal network architecture by removing redundant hidden structures.

\bibliography{iclr2021_conference}
\bibliographystyle{iclr2021_conference}

\appendix

\section{Projection Region}\label{appendix:projection_region}

In this Appendix, we derive the projection region of~\algacro{}, and reveal that is a superset of those of~\proxsg{},~\proxsvrg{} and~\proxspider{} under the same $\alpha_k$ and $\lambda$.
\begin{proposition}
The~\halfspacestep{} of~\algacro{} yields next iterate $x_{k+1}$ based on the trial iterate $\hat{x}_{k+1}=x_k-\alpha_k\Grad f_{\mathcal{B}_k}(x_k)$ as follows for each $g\in\mathcal{I}^{\neq 0}(x_k)$
\begin{equation}
[x_{k+1}]_g=
\begin{cases}
[\hat{x}_{k+1}]_g-\alpha_k\lambda \frac{[x_k]_g}{\norm{[x_k]_g}}& \text{if}\ [\hat{x}_{k+1}]_g^\top[x_k]_g> (\alpha_k\lambda + \epsilon)\norm{[x_k]_g} \\
0 & \text{otherwise}.
\end{cases}
\end{equation}
Consequently, if $\norm{[\hat{x}_{k+1}]_g}\leq \alpha_k\lambda$,  then $[x_{k+1}]_g=0$ for any $\epsilon\geq 0$. 
\end{proposition}
\begin{proof}
For $g\in\I^{\neq 0}(x_k)\bigcap\mathcal{I}^{\neq 0}(x_{k+1})$, by Algorithm~\ref{alg:main.x.halfspacestep}, it is equivalent to  
\begin{equation}
\begin{split}
\left[x_k-\alpha_k\Grad f_{\mathcal{B}_k}(x_k)-\alpha_k\lambda \frac{[x_k]_g}{\norm{[x_k]_g}}\right]_g^\top[x_k]_g> \epsilon \norm{[x_k]_g}^2,\\
[\hat{x}_{k+1}]_g^\top[x_k]_g -\alpha_k\lambda \norm{[x_k]_g} > \epsilon \norm{[x_k]_g}^2,\\
[\hat{x}_{k+1}]_g^\top[x_k]_g > (\alpha_k\lambda+\epsilon\norm{[x_k]_g})\norm{[x_k]_g}.
\end{split}
\end{equation}
Similarly, $g\in\I^{\neq 0}(x_k)\bigcap\mathcal{I}^{0}(x_{k+1})$ is equivalent to 
\begin{equation}\label{eq:xkp1_equals_zero}
\begin{split}
\left[x_k-\alpha_k\Grad f_{\mathcal{B}_k}(x_k)-\alpha_k\lambda \frac{[x_k]_g}{\norm{[x_k]_g}}\right]_g^\top[x_k]_g\leq \epsilon \norm{[x_k]_g}^2,\\
[\hat{x}_{k+1}]_g^\top[x_k]_g -\alpha_k\lambda \norm{[x_k]_g} \leq \epsilon \norm{[x_k]_g}^2,\\
[\hat{x}_{k+1}]_g^\top[x_k]_g \leq (\alpha_k\lambda+\epsilon\norm{[x_k]_g})\norm{[x_k]_g}.
\end{split}    
\end{equation}
If $\norm{[\hat{x}_{k+1}]_g}\leq \alpha_k\lambda$, then 
\begin{equation}
[\hat{x}_{k+1}]_g^\top[x_k]_g \leq \norm{[\hat{x}_{k+1}]_g}\norm{[x_k]_g}\leq \alpha_k\lambda \norm{[x_k]_g}.
\end{equation}
Hence $[x_{k+1}]_g=0$ holds for any $\epsilon\geq 0$ by~\eqref{eq:xkp1_equals_zero}, which implies that the projection region of~\proxsg{} and its variance reduction variants, \eg,~\proxsvrg{},~\proxspider{} and~\saga{} are the subsets of~\algacro{}'s.
\end{proof}

\section{Non-Lipschitz Continuity of $\Grad \Psi(x)$ on $\mathbb{R}^n$}\label{sec.Psi_not_lipschitz_countinuous_gradient}\label{appendix:nonlipschitz-continuity}

The first-derivative of $\Psi(x)$ at $x\neq 0$ can be written as 
\begin{equation}
\Grad \Psi(x)=\Grad f(x)+\lambda\sum_{g\in\mathcal{G}}\frac{[x]_g}{\norm{[x]_g}}
\end{equation}
We next show $\frac{[x]_g}{\norm{[x]_g}}$ is not Lipschitz continuous on $\mathbb{R}^n$ if $|g|\geq 2$. Take a example for $[x]_g=(x_1, x_2)^\top\in\mathbb{R}^2$, and select $x_1=(t, a_1t), x_2=(t, a_2t), a_1\neq a_2$ and $t\in\mathbb{R}$. Then suppose there exists a positive constant $L<\infty$ such that Lipschitz continuity holds as follows
\begin{equation}
\begin{split}
\norm{\frac{x_1}{\norm{x_1}}-\frac{x_2}{\norm{x_2}}}&\leq L\norm{x_1-x_2}\\
\norm{\frac{(1,a_1)}{\sqrt{1+a_1^2}}-\frac{(1,a_2)}{\sqrt{1+a_2^2}}}&\leq L|a_1-a_2|\cdot|t|
\end{split}
\end{equation}
holds for any $t\in\mathbb{R}$, and note the left hand side is a positive constant. However, letting $t\to 0$, we have that $L\to \infty$ which contradicts the $L<\infty$. Therefore, $\frac{[x]_g}{\norm{[x]_g}}$ is not Lipschitz continuous on $\mathbb{R}^2$, specifically the region surrounding the origin point. 

Although $[\Grad \Psi(x)]_{\I^{\neq 0}(x)}$ is not Lipscthiz continuous on $\mathbb{R}^{n}$, the Lipschitz continuity still holds on by excluding a fixed size $\ell_2$-ball centered at the origin for the group of non-zero variables $\I^{\neq 0}(x)$ from $\mathbb{R}^{n}$. For our paper, we define the region where Lipscthiz continuity of $[\Grad \Psi(x)]_{\I^{\neq 0}(x)}$ still holds as 
\begin{equation}\label{def:mathcal_X}
\mathcal{X}=\{x: \norm{[x]_g}\geq \delta_1\ \text{for each}\ g\in\I^{\neq 0}(x),\ \text{and}\ [x]_g=0\ \text{for each}\ g\in\I^{0}(x) \}.    
\end{equation}

\section{Convergence Analysis Proof}\label{appendix:convergence_analysis}

Denote the following frequently used constant $R$ describing the size of neighbor around $x^*$.
\begin{equation}\label{def:R}
R:=\min\left\{\frac{-(\delta_1+2\epsilon\delta_2)+\sqrt{(\delta_1+2\epsilon\delta_2)^2-4\epsilon^2\delta_2+4\epsilon\delta_1^2}}{\epsilon},\delta_1\right\}>0.
\end{equation}
\textbf{Remark:}~\eqref{def:R} is well defined as $0<\epsilon< \frac{\delta_1^2}{\delta_2}$, and degenerated to $\delta_1$ as $\epsilon=0$.

\subsection{Sufficient Decrease of~\proxsgstep{} and~\halfspacestep{}}\label{appendix:proof_sufficient_decrease}

Our convergence analysis relies on the following sufficient decrease properties of~\halfspacestep{} and~\proxsgstep{}.

\paragraph{Sufficient Decrease of~\halfspacestep{} as Lemma~\ref{lemma:sufficient_decrease_half_space}:} 

\begin{proof}

It follows Algorithm~\ref{alg:main.x.halfspacestep} and the definition of $\tilde{\G}_k$ and $\hat{\G}_k$ that $x_{k+1}=x_k+\alpha_kd_k$ where $d_k$ is
\begin{equation}\label{eq:proof_d_k_def}
[d_k]_g=
\begin{cases}
-[\partial \Psi_{\mathcal{B}_k}(x_{k})]_g& \text{if}\ g\in\tilde{\mathcal{G}}_k=\I^{\neq 0}(x_k)\bigcap \I^{\neq 0}(x_{k+1}),\\
-[x_{k}]_g/\alpha_k & \text{if}\ g \in \hat{\mathcal{G}}_k=\I^{\neq 0}(x_k)\bigcap \I^{0}(x_{k+1}),\\
0 & \text{otherwise}.
\end{cases}
\end{equation}
We also notice that for any $g\in\hat{\mathcal{G}}_k$, the following holds
\begin{equation}\label{eq:descent_direction_tmp1}
\begin{split}
[x_k-\alpha_k\partial \Psi_{\mathcal{B}_k}(x_k)]_g^\top[x_k]_g<\epsilon \norm{[x_k]_g}^2,\\
(1-\epsilon)\norm{[x_k]_g}^2< \alpha_k[\partial \Psi_{\mathcal{B}_k}(x_k)]_g^\top[x_k]_g.
\end{split}
\end{equation}
For simplicity, let $\I^{\neq 0}_k:=\I^{\neq 0}(x_k)$. Since $[d_k]_g=0$ for any $g\in \I^{0}(x_k)$, then by~(\ref{eq:proof_d_k_def}) and~(\ref{eq:descent_direction_tmp1}), we have 
\begin{equation}
\begin{split}
d_k^\top\partial \Psi_{\B_k}(x_k)&=[d_k]_{\I^{\neq 0}_k}^\top[\partial \Psi_{\B_k}(x_k)]_{\I^{\neq 0}_k}\\
&=-\sum_{g\in\tilde{\G}_k}\norm{[\partial \Psi_{\B_k}(x_k)]_g}^2-\sum_{g\in \hat{\G}_k}\frac{1}{\alpha_k}[x_k]_g^\top[\partial \Psi_{\B_k}(x_k)]_g\\
&\leq -\sum_{g\in\tilde{\G}_k}\norm{[\partial \Psi_{\B_k}(x_k)]_g}^2-\sum_{g\in \hat{\G}_k}\frac{1}{\alpha_k^2}(1-\epsilon)\norm{[x_k]_g}^2< 0,
\end{split}
\end{equation}
holds for any $\epsilon\in[0,1)$, which implies that $d_k$ is a descent direction for $\Psi_{\B_k}(x_k)$. 

Now, we start to prove the suffcient decrease of~\halfspacestep{}. By the descent lemma, $x_k\in\mathcal{X}$ and the Lipschitz continuity of $[\partial \Psi_{\mathcal{B}_k}]_{\I_k^{\neq 0}}$ on $\mathcal{X}$, we have that 
\begin{equation}\label{eq:decentlemma}
\Psi_{\mathcal{B}_k}(x_{k}+\alpha_k d_{k})\leq \Psi_{\mathcal{B}_k}(x_{k})+\alpha_k[\partial \Psi_{\mathcal{B}_k}(x_{k})]_{\I_k^{\neq 0}}^\top[d_{k}]_{\I_k^{\neq 0}}+\frac{L}{2}\alpha_k^2\norm{[d_{k}]_{\I_k^{\neq 0}}}^2.
\end{equation}
Then it follows~\eqref{eq:proof_d_k_def} that~\eqref{eq:decentlemma} can be rewritten as follows
\begin{equation}\label{eq:decentlemma_more}
\begin{split}
&\Psi_{\mathcal{B}_k}(x_{k}+\alpha_k d_{k})\\
\leq& \Psi_{\mathcal{B}_k}(x_{k})+\alpha_k[\partial \Psi_{\mathcal{B}_k}(x_{k})]_{\I^{\neq 0}_k}^\top[d_{k}]_{\I^{\neq 0}_k}+\frac{L}{2}\alpha_k^2\norm{[d_{k}]_{\I^{\neq 0}_k}}^2\\
=&\Psi_{\mathcal{B}_k}(x_{k})-\sum_{g\in\tilde{\G}_k}\norm{[\partial \Psi_{\mathcal{B}_k}(x_{k})]_g}^2\left(\alpha_k-\frac{L}{2}\alpha_k^2\right)-\sum_{g\in\hat{\G}_k} \left\{[\partial \Psi_{\mathcal{B}_k}(x_{k})]_g^\top[x_{k}]_g-\frac{L}{2}\norm{[x_{k}]_g}^2\right\}
\end{split}
\end{equation}
Consequently, combining with $\epsilon\in [0,1)$ and~\eqref{eq:descent_direction_tmp1}, ~\eqref{eq:decentlemma_more} can be further shown as 
\begin{equation}
\Psi_{\mathcal{B}_k}(x_{k+1})\leq \Psi_{\mathcal{B}_k}(x_{k})-\left(\alpha_k-\frac{\alpha_k^2L}{2}\right)\sum_{g\in\tilde{\G}_k}\norm{[\partial \Psi_{\mathcal{B}_k}(x_{k})]_g}^2-\left(\frac{1-\epsilon}{\alpha_k}-\frac{L}{2}\right)\sum_{g\in\hat{\G}_k}\norm{[x_{k}]_g}^2,
\end{equation}
which completes the proof.
   
\end{proof}

\paragraph{Sufficient Decrease of~\proxsgstep:} The second lemma is well known for proximal operator under our notations. We include this proof for completeness.
\begin{lemma}\label{lemma:Psi_decrease_proxsg}
	Line~\ref{line:prox} of Algorithm~\ref{alg:main.x.prox_sg_step} yields that $x_{k+1}=x_{k}-\alpha_k\xi_{\alpha_k, \mathcal{B}_k}(x_k)$, where
	\begin{equation}
	\xi_{\alpha_k, \mathcal{B}_k}(x_k)\in -\left(\Grad f_{\mathcal{B}_k}(x_{k})+\lambda \partial \Omega(x_{k+1})\right).
	\end{equation}
	And the objective value $ \Psi_{\mathcal{B}_k} $ satisfies
	\begin{equation}\label{eq:decrease_p_batch}
	\Psi_{\mathcal{B}_k}(x_{k+1})\leq \Psi_{\mathcal{B}_k}(x_{k})-\left(\alpha_k-\frac{\alpha_k^2L}{2}\right)\norm{\xi_{\alpha_k, \mathcal{B}_k}(x_k)}^2.
	\end{equation}	
\end{lemma}	
\begin{proof}
	It follows from the line~\eqref{line:prox} in Algorithm~\ref{alg:main.x.prox_sg_step} and the definitions of proximal operator that
	\begin{equation}
	\begin{split}
	\quad x_{k+1}&=\argmin_{x\in \mathbb{R}^n}\ \frac{1}{2\alpha_k}\norm{x-(x_k-\alpha_k\Grad f_{\mathcal{B}_k}(x_k))}^2+\lambda\Omega(x)\\
	&=\argmin_{x\in \mathbb{R}^n}\ \Grad f_{\mathcal{B}_k}(x_k)^\top(x-x_k)+\lambda\Omega(x)+\frac{1}{2\alpha_k}\norm{x-x_k}^2
	\end{split}
	\end{equation}
	By the optimal condition, we have 
	\begin{equation}
	\begin{split}
	0 & \in \frac{1}{\alpha_k}(x_{k+1}-x_k)+\Grad f_{\mathcal{B}_k}(x_k)+\lambda\partial\Omega(x_{k+1}).
	\end{split}
	\end{equation}
	Since $ x_{k+1}=x_k-\alpha_k\xi_{\alpha_k,\mathcal{B}_k}(x_k)$, we have
	\begin{equation}
	0 \in -\xi_{\alpha_k,\mathcal{B}_k}(x_k)+\Grad f_{\mathcal{B}_k}(x_k)+\lambda\partial\Omega(x_{k+1}),
	\end{equation}
	which implies that
	\begin{equation}
	\xi_{\alpha_k,\mathcal{B}_k}(x_k)\in \Grad f_{\mathcal{B}_k}(x_k)+\lambda \partial\Omega(x_{k+1}).
	\end{equation}
	And thus there exists some $v\in \partial \Omega(x_{k+1})$ such that 
	\begin{equation}\label{eq:st}
	\xi_{\alpha_k,\mathcal{B}_k}(x_k)= \Grad f_{\mathcal{B}_k}(x_k) +\lambda v.
	\end{equation}
	\noindent
	By Lipschitz continuity of $ \Grad f_{\mathcal{B}_k} $ and convexity of $\Omega(\cdot)$, we have 
	\begin{equation}\label{eq:f_lipschiz_convex}
	\begin{split}
	f_{\mathcal{B}_k}(x_{k+1})&=f_{\mathcal{B}_k}(x_k-\alpha_k\xi_{\alpha_k,\mathcal{B}_k}(x_k))\\
	&\leq f_{\mathcal{B}_k}(x_k)-\alpha_k\Grad f_{\mathcal{B}_k}(x_k)^\top\xi_{\alpha_k,\mathcal{B}_k}(x_k)+\frac{\alpha_k^2L}{2}\norm{\xi_{\alpha_k,\mathcal{B}_k}(x_k)}^2
	\end{split}
	\end{equation}
	and 
	\begin{equation}\label{eq:omega_convex}
	\begin{split}
	\lambda\Omega(x_{k+1})&=\lambda\Omega(x_k-\alpha_k\xi_{\alpha_k,\mathcal{B}_k}(x_k))\\
	&\leq \lambda \Omega(x_k) + \lambda v^\top(x_k-\alpha_k\xi_{\alpha_k,\mathcal{B}_k}(x_k)-x_k)\\
	&=\lambda \Omega(x_k) - \alpha_k\lambda v^\top\xi_{\alpha_k,\mathcal{B}_k}(x_k).
	\end{split}
	\end{equation}		
	Hence, by \eqref{eq:st}, \eqref{eq:f_lipschiz_convex}  and~\eqref{eq:omega_convex}, the objective $\Psi_{\mathcal{B}_k}(x_{k+1})$ satisfies
	\begin{equation*}\label{eq:Psi_decrease}
	\begin{split}
	&\Psi_{\mathcal{B}_k}(x_{k+1})=f_{\mathcal{B}_k}(x_{k+1}) + \lambda\Omega(x_{k+1})\\
	\leq& f_{\mathcal{B}_k}(x_k)-\alpha_k\Grad f_{\mathcal{B}_k}(x_k)^\top\xi_{\alpha_k,\mathcal{B}_k}(x_k)+\frac{\alpha_k^2L}{2}\norm{\xi_{\alpha_k,\mathcal{B}_k}(x_k)}^2+\lambda \Omega(x_k) - \alpha_k\lambda v^\top\xi_{\alpha_k,\mathcal{B}_k}(x_k)\\
	=&\Psi_{\mathcal{B}_k}(x_k)-\alpha_k(\Grad f_{\mathcal{B}_k}(x_k)+\lambda v)^\top\xi_{\alpha_k,\mathcal{B}_k}(x_k)+\frac{\alpha_k^2L}{2}\norm{\xi_{\alpha_k,\mathcal{B}_k}(x_k)}^2\\
	=&\Psi_{\mathcal{B}_k}(x_k) -\left(\alpha_k-\frac{\alpha_k^2L}{2}\right)\norm{\xi_{\alpha_k,\mathcal{B}_k}(x_k)}^2,
	\end{split}
	\end{equation*}
	which completes the proof.
\end{proof}

According to Lemma~\ref{lemma:sufficient_decrease_half_space} and Lemma~\ref{lemma:Psi_decrease_proxsg}, the objective value on a mini-batch tends to achieve a sufficient decrease in both Prox-SG Step and Half-Space Step given $\alpha_k$ is small enough. By taking the expectation on both sides, we obtain the following result characterizing the sufficient decrease from $\Psi(x_{k})$ to $\mathbb{E}\left[\Psi(x_{k+1})\right]$.

\begin{corollary}\label{corollary:Psi_epoch_decrease}	
	For iteration $k$, we have 
	\begin{enumerate}[label=(\roman*),leftmargin=0.5cm]
		\item 	if $k$th iteration conducts~\proxsgstep,  then
		\begin{equation}
		\mathbb{E}\left[\Psi(x_{k+1})\right]\leq \Psi(x_k)-\left(\alpha_k-\frac{\alpha_k^2L}{2}\right)\mathbb{E}\left[\norm{\xi_{\alpha_k,\mathcal{B}_k}(x_k)}^2\right].
		\end{equation}
		\item if $k$th iteration conducts~\halfspacestep, $x_k\in\mathcal{X}$, then
		\begin{equation}
		\mathbb{E}\left[\Psi(x_{k+1})\right]\leq \Psi(x_k)-\sum_{g\in\tilde{\mathcal{G}}_k}\left(\alpha_k-\frac{\alpha_k^2L}{2}\right)\mathbb{E}\left[\norm{\partial \Psi_{\mathcal{B}_k}(x_k)}^2\right]-\left(\frac{1-\epsilon}{\alpha_k}-\frac{L}{2}\right)\sum_{g\in\hat{\G}_k}\norm{[x_{k}]_g}^2.
		\end{equation}
	\end{enumerate}
\end{corollary}

Corollary~\ref{corollary:Psi_epoch_decrease} shows that the bound of $\Psi$ depends on step size $\alpha_k$ and norm of search direction. It further indicates that both \halfspacestep{} and \proxsgstep{} can make some progress to optimality with proper selection of $\alpha_k$.

\subsection{Proof of Theorem~\ref{thm:convergence}}\label{appendix:convergence_theorem}

Toward that end, we first show that if the optimal distance from $x_k$ to the local minimizer $x^*$ is sufficiently small, then~\algacro{} already covers the supports of $x^*$,~\ie, $\I^{\neq 0}(x^*)\subseteq \I^{\neq 0}(x_k)$. 
\begin{lemma}\label{lemma:support_cover}
If $\norm{x_k-x^*}\leq R$, then $\I^{\neq 0}(x^*)\subseteq \I^{\neq 0}(x_k)$.    
\end{lemma}
\begin{proof}
For any $g\in I^{\neq 0}(x^*)$, by the assumption of this lemma and the definition of $R$ as~\eqref{def:R} and $\delta_1$ in Assumption~\ref{assumption:delta}, we have that 
\begin{equation}
\begin{split}
\norm{[x^*]_g}-\norm{[x_k]_g}&\leq \norm{[x_k-x^*]_g}\leq \norm{x_k-x^*}\leq R\leq \delta_1\\
\norm{[x_k]_g}&\geq \norm{[x^*]_g}-\delta_1\geq 2\delta_1-\delta_1=\delta_1>0
\end{split}
\end{equation}
Hence $\norm{[x_k]_g}\neq 0$, \ie, $g\in \I^{\neq 0}(x_k)$. Therefore, $\I^{\neq 0}(x^*)\subseteq \I^{\neq 0}(x_k)$.
\end{proof}

The next lemma shows that if the distance between current iterate $x_k$ and $x^*$, \ie, $\norm{x_k-x^*}$ is sufficiently small, then $x^*$ inhabits the reduced space $\S_k:=\S(x_k)$.
\begin{lemma}\label{lemma:x_star_in_polyhedron} 
	Under Assumption~\ref{assumption}, if $0\leq \epsilon<\frac{\delta_1^2}{\delta_2}$,  $\norm{x_{k}-x^*}\leq R$, then for each $g\in\mathcal{I}^{\neq 0}(x^*)$, 
	\begin{equation}
	[x_{k}]_g^\top[x^*]_g\geq \epsilon\norm{[x_k]_g}^2
	\end{equation} 
	Consequently, it implies $x^*\in\S_k$ by the definition as~\eqref{def:polytope}.
\end{lemma}
\begin{proof}
	It follows the assumption of this lemma and the definition of $R$ in~\eqref{def:R}, $\delta_1$ and $\delta_2$ in Assumption~\ref{assumption:delta} that for any  $g\in\mathcal{I}^{\neq 0}(x^*)$, 
	\begin{equation}
	\begin{split}
	\norm{[x_k]_g}\leq \norm{[x^*]_g} +R\leq 2\delta_2+R,
	\end{split}
	\end{equation}
	and the $\left[-(\delta_1+2\epsilon\delta_2)+\sqrt{(\delta_1+2\epsilon\delta_2)^2-4\epsilon^2\delta_2+4\epsilon\delta_1^2}\right]/\epsilon$ in~\eqref{def:R} is actually the solution of $\epsilon z^2+(4\epsilon\delta_2+2\delta_1)z+4\epsilon\delta_2^2-4\delta_1^2=0$ regarding $z\in \mathbb{R}^+$. Then we have that

	\begin{equation}
	\begin{split}
	[x_{k}]_g^\top[x^*]_g=&[x_{k}-x^*+x^*]_g^\top [x^*]_g\\
	=&[x_{k}-x^*]_g^\top[x^*]_g+\norm{[x^*]_g}^2\\
	\geq& \norm{[x^*]_g}^2-\norm{[x_k-x^*]_g}\norm{[x^*]_g}\\
	=& \norm{[x^*]_g}(\norm{[x^*]_g}-\norm{[x_k-x^*]_g})\\
	\geq & 2\delta_1(2\delta_1-R)\geq \epsilon (2\delta_2+R)^2\\
	\geq &\epsilon\norm{[x_k]_g}^2
	\end{split}
	\end{equation}
	holds for any $g\in\mathcal{I}^{\neq 0}(x^*)$, where the second last inequality holds because that $2\delta_1(2\delta_1-R)=\epsilon(2\delta_2+R)^2$ as $R=\left[-(\delta_1+2\epsilon\delta_2)+\sqrt{(\delta_1+2\epsilon\delta_2)^2-4\epsilon^2\delta_2+4\epsilon\delta_1^2}\right]/\epsilon$. Now combing with the definition of $\S_k$ as~\eqref{def:polytope}, we have $x^*$ inhabits $\S_k$, which completes the proof.
\end{proof}

Furthermore, if $\norm{x_k-x^*}$ is small enough and the step size is selected properly, every recovery of group sparsity by~\halfspacestep{} can be guaranteed as successful as stated in the following lemma.
\begin{lemma}\label{lemma.project_as_zero_group}
	Suppose $k\geq N_\P$,  $\norm{x_{k}-x^*}\leq R$, $0\leq \epsilon<\frac{2\delta_1-R}{2\delta_2+R}$ and $0<\alpha_k\leq\frac{2\delta_1-R-\epsilon(2\delta_2+R)}{M}$, then for any $g\in\hat{\G}_k={\I^{\neq 0}(x_k)}\bigcap \I^0(x_{k+1})$, $g$ must be in $\I^{0}(x^*)$, \ie, $g\in\I^0(x^*)$.
\end{lemma}
\begin{proof}
To prove it by contradiction, suppose there exists some $g\in\hat{\G}_k$ such that $g\in \I^{\neq 0}(x^*)$. Since $g\in\hat{\G}_k={\I^{\neq 0}(x_k)}\bigcap \I^0(x_{k+1})$, then the group projection~\eqref{def:proj} is trigerred at $g$ such that
\begin{equation}\label{eq:hypothesis}
\begin{split}
[\tilde{x}_{k+1}]_{g}^\top[x_k]_{g}&=[x_k-\alpha\Grad \Psi_{\B_k}(x_k)]_{g}^\top[x_k]_{g}\\
&=\norm{[x_k]_{g}}^2-\alpha_k [\Grad \Psi_{\B_k}(x_k)]_{g}^\top [x_k]_{g}< \epsilon \norm{[x_k]_g}^2.
\end{split}
\end{equation}	
On the other hand, it follows the assumption of this lemma and $g\in\I^{\neq 0}(x^*)$ that 
\begin{equation}
 \norm{[x_{k}-x^*]_{g}}\leq \norm{x_k-x^*}\leq R 
\end{equation}
Combining the definition of $\delta_1$ and $\delta_2$, we have that 
\begin{equation}
\begin{split}
\norm{[x_k]_{g}}\geq \norm{[x^*]_{g}}-R\geq  2\delta_1 -R\\
\norm{[x_k]_{g}}\leq \norm{[x^*]_{g}}+R\leq  2\delta_2 +R\\
\end{split}
\end{equation}
It then follows $0<\alpha_k\leq\frac{2\delta_1-R-\epsilon(2\delta_2+R)}{M}$, where note $2\delta_1-R-\epsilon(2\delta_2+R)>0$ as $R\leq \delta_1$ and $\epsilon<\frac{2\delta_1-R}{2\delta_2+R}$, that 
\begin{equation}
\begin{split}
[\tilde{x}_{k+1}]_{g}^\top[x_k]_{g}&=\norm{[x_k]_{g}}^2-\alpha_k [\Grad \Psi_{\B_k}(x_k)]_{g}^\top [x_k]_{g}\\
&\geq \norm{[x_k]_{g}}^2-\alpha_k\norm{[\Grad \Psi_{\B_k}(x_k)]_g}\norm{[x_k]_g}\\
&= \norm{[x_k]_{g}}(\norm{[x_k]_{g}}-\alpha_k\norm{[\Grad \Psi_{\B_k}(x_k)]_g})\\
&\geq  \norm{[x_k]_{g}}(\norm{[x_k]_{g}}-\alpha_kM)\\
&\geq  \norm{[x_k]_{g}}\left[(2\delta_1-R)-\alpha_kM\right]\\
&\geq \norm{[x_k]_{g}}\left[(2\delta_1-R)-\frac{2\delta_1-R-\epsilon(2\delta_2+R)}{M}M\right]\\ 
&\geq \norm{[x_k]_{g}}\left[(2\delta_1-R)-2\delta_1+R+\epsilon(2\delta_2+R)\right]\\ 
&\geq \epsilon\norm{[x_k]_{g}}(2\delta_2+R)\\
&\geq \epsilon\norm{[x_k]_g}^2
\end{split}
\end{equation}
which contradicts with~\eqref{eq:hypothesis}. Hence, we conclude that any $g$ of variables projected to zero, \ie, $g\in\hat{\G}_k={\I^{\neq 0}(x_k)}\bigcap \I^0(x_{k+1})$ are exactly also the zeros on the optimal solution $x^*$, \ie, $g\in\I^{0}(x^*)$. 
\end{proof}

We next present that if the iterate of~\halfspacestep{} is close enough to the optimal solution $x^*$, then $x^*$ inhabits all  reduced spaces constructed by the subsequent iterates of~\halfspacestep{} with high probability. To establish this results, we require the below two lemmas. The first bounds the accumulated error because of random sampling.    

\begin{lemma}\label{lemma:convergence-series}
Given any $\theta > 1$, $K\geq N_\P$, let $k:=K+t$, $t\in\mathbb{Z}^+\bigcup \{0\}$, then there exists $\alpha_k=\O(1/t)$ and $|\B_k|=\O(t)$, such that for any $y_t\in \mathbb{R}^n$,  
\begin{align*}
    \max_{\{y_t\}_{t = 0}^{\infty} \in \mathcal{X}^{\infty}} \sum_{t = 0}^{\infty} \alpha_{k} \|e_{\mathcal{B}_{k}}(y_{t})\|_2 \leq \frac{3R^2}{8(4R + 1)}  
\end{align*}
holds with probability at least $1 - \frac{1}{\theta^2}$.
\end{lemma}
\begin{proof}
Define random variable $Y_t := \alpha_{K + t} \|e_{\mathcal{B}_{K + t}}(y_{t})\|_2$ for all $t \geq 0$. Since $\{y_t\}_{t = 0}^{\infty}$ are arbitrarily chosen, then the random variables $\{Y_t\}_{t = 0}^{\infty}$ are independent. Let $Y := \sum_{t= 0}^{\infty} Y_t$. Using Chebshev's inequality, we obtain
\begin{align}\label{eq:chevshev_inequality}
    \mathbb{P}\left( Y \geq \mathbb{E}[Y] + \theta \sqrt{\text{Var}[Y]} \right) \leq \mathbb{P}\left( |Y - \mathbb{E}[Y]| \geq \theta \sqrt{\text{Var}[Y]} \right) \leq \frac{1}{\theta^2}. 
\end{align}
And based on the Assumption~\ref{assumption}, there exists an upper bound $\sigma^2>0$ for the variance of random noise $e(x)$ generated from the one-point mini-batch, \ie, $\mathcal{B}=\{i\}, i = 1,\ldots, N$. Consequently, for each $t \geq 0$, we have $\mathbb{E}[Y_t] \leq \frac{\alpha_{K + t} \sigma}{\sqrt{|\mathcal{B}_{K + t}|}}$ and $\text{Var}[Y_t] \leq \frac{\alpha_{K + t}^2 \sigma^2}{|\mathcal{B}_{K + t}|}$, then combining with~\eqref{eq:chevshev_inequality}, we have
\begin{align}
    Y &\leq \mathbb{E}[Y] + \theta \sqrt{\text{Var}[Y]} \\
    &\leq \sum_{t = 0}^{\infty} \frac{\alpha_{K + t} \sigma}{\sqrt{|\mathcal{B}_{k + t}|}} + \theta \cdot \sum_{t = 0}^{\infty} \frac{\alpha_{K + t}^2 \sigma^2}{|\mathcal{B}_{K + t}|}\\
    &\leq \sum_{t = 0}^{\infty} \frac{\alpha_{K + t} \sigma}{\sqrt{|\mathcal{B}_{k + t}|}} + \theta \cdot \sum_{t = 0}^{\infty} \frac{\alpha_{K + t} \sigma}{\sqrt{|\mathcal{B}_{K + t}|}} =(1+\theta)\sum_{t = 0}^{\infty} \frac{\alpha_{K + t} \sigma}{\sqrt{|\mathcal{B}_{K + t}|}}
\end{align}
holds with probability at least $1 - \frac{1}{\theta^2}$. Here, for the second inequality, we use the property that the equality $\mathbb{E}[\sum_{t = 0}^{\infty} Y_i] = \sum_{t = 0}^{\infty} \mathbb{E}[ Y_i]$ holds whenever $\sum_{t = 0}^{\infty} \mathbb{E}[|Y_i|]$ convergences, see Section 2.1 in \cite{mitzenmacher2005probability}; and for the third inequality, we use $\frac{\alpha_{K + t} \sigma}{\sqrt{|\mathcal{B}_{K + t}|}}\leq 1$ without loss of generality as the common setting of large mini-batch size and small step size.  

Given any $\theta > 1$, there exists some $\alpha_{k}=\O(1/t)$ and $|\B_{k}|=\O(t)$, the above series converges and satisfies that 
\begin{align*}
    (1+\theta)\sum_{t = 0}^{\infty} \frac{\alpha_{K + t} \sigma}{\sqrt{|\mathcal{B}_{K + t}|}} \leq \frac{3R^2}{8(4R + 1)}
\end{align*}
holds. Notice that the above proof holds for any given sequence $\{y_t\}_{t = 0}^{\infty} \in \mathcal{X}^{\infty}$, thus  
\begin{align*}
    \max_{\{y_t\}_{t = 0}^{\infty} \in \mathcal{X}^{\infty}} \sum_{t = 0}^{\infty} \alpha_{k} \|e_{\mathcal{B}_{k}}(y_{t})\|_2 \leq \frac{3R^2}{8(4R + 1)}
\end{align*}
holds with probability at least $1 - \frac{1}{\theta^2}$. 
\end{proof}

The second lemma draws if previous iterate of~\halfspacestep{} falls into the neighbor of $x^*$, then under appropriate step size and mini-batch setting, the current iterate also inhabits the neighbor with high probability. 

\begin{lemma}\label{lemma:k_plus_1_optimal_dist_non_increase}
	Under the assumptions of Lemma~\ref{lemma:convergence-series}, suppose $\norm{x_{K}-x^*}\leq R/2$; for any $\ell$ satisfying $K\leq \ell<K+ t$, $0<\alpha_{\ell}\leq \min\{\frac{1}{L},\frac{2\delta_1-R-\epsilon(2\delta_2+R)}{M}\}$, $|B_\ell|\geq N-\frac{N}{2M}$ and $\norm{x_{\ell}-x^*}\leq R$ holds, then 
	\begin{equation}
	\norm{x_{K+t}-x^*}\leq R.
	\end{equation}
	holds with probability at least $1-\frac{1}{\theta^2}$.
\end{lemma}
\begin{proof}
It follows the assumptions of this lemma, Lemma~\ref{lemma.project_as_zero_group}  that for any $\ell$ satisfying $K\leq \ell<K+ t$
\begin{equation}
\norm{[x^*]_g}=0,\ \text{for any}\ g\in \hat{\mathcal{G}}_{\ell}.
\end{equation}
Hence we have that for $K\leq \ell<K+ t$,
\begin{equation}\label{eq:optimal_dist_ell}
\begin{split}
&\norm{x_{\ell+1}-x^*}^2\\
=&\sum_{g\in\tilde{\mathcal{G}}_\ell}\norm{[x_{\ell}-x^*-\alpha_\ell\Grad \Psi(x_{\ell})-\alpha_\ell e_{\B_\ell}(x_\ell)]_g}^2+\sum_{g\in\hat{\mathcal{G}}_k}\norm{[x_{\ell}-x^*-x_{\ell}]_g}^2\\
=&\sum_{g\in\tilde{\mathcal{G}}_\ell}\left\{\norm{[x_{\ell}-x^*]_g}^2-2\alpha_\ell[x_{\ell}-x^*]_g^\top[\Grad \Psi(x_{\ell})+e_{\B_\ell}(x_\ell)]_g+\alpha_\ell^2\norm{[\Grad \Psi(x_{\ell})+e_{\B_\ell}(x_\ell)]_g}^2\right\}+\sum_{g\in\hat{\G}_\ell}\norm{[x^*]_g}^2\\
=&\sum_{g\in\tilde{\mathcal{G}}_\ell}\left\{\norm{[x_{\ell}-x^*]_g}^2-2\alpha_\ell[x_{\ell}-x^*]_g^\top[\Grad \Psi(x_{\ell})]_g-2\alpha_\ell[x_{\ell}-x^*]_g^\top[e_{\B_\ell}(x_\ell)]_g+\alpha_\ell^2\norm{[\Grad \Psi(x_{\ell})+e_{\B_\ell}(x_\ell)]_g}^2\right\}\\
\leq&\sum_{g\in\tilde{\G}_\ell}\norm{[x_{\ell}-x^*]_g}^2-\norm{[\Grad \Psi(x_{\ell})]_g}^2\left(2\frac{\alpha_\ell}{L}-\alpha_\ell^2\right)-2\alpha_\ell[x_{\ell}-x^*]_g^\top[e_{\B_\ell}(x_\ell)]_g+\alpha_\ell^2\norm{[e_{\B_\ell}(x_\ell)]_g}^2\\
&+2\alpha_\ell^2[\Grad \Psi(x_\ell)]_g^\top[e_{\B_\ell}(x_\ell)]_g\\
\leq& \sum_{g\in\tilde{\G}_\ell}\norm{[x_{\ell}-x^*]_g}^2-\norm{[\Grad \Psi(x_{\ell})]_g}^2\left(2\frac{\alpha_\ell}{L}-\alpha_\ell^2\right) +2\alpha_\ell\norm{[x_{\ell}-x^*]_g}\norm{[e_{\B_\ell}(x_\ell)]_g}+\alpha_\ell^2\norm{[e_{\B_\ell}(x_\ell)]_g}^2\\
&+2\alpha_\ell^2\norm{[\Grad \Psi(x_\ell)]_g}\norm{[e_{\B_\ell}(x_\ell)]_g}\\
\leq& \sum_{g\in\tilde{\G}_\ell}\norm{[x_{\ell}-x^*]_g}^2-\norm{[\Grad \Psi(x_{\ell})]_g}^2\left(2\frac{\alpha_\ell}{L}-\alpha_\ell^2\right) +(2\alpha_\ell+2\alpha_\ell^2L)\norm{[x_{k}-x^*]_g}\norm{[e_{\B_\ell}(x_\ell)]_g}+\alpha_\ell^2\norm{[e_{\B_\ell}(x_\ell)]_g}^2\\
\leq& \sum_{g\in\tilde{\G}_\ell}\left\{\norm{[x_{\ell}-x^*]_g}^2-\norm{[\Grad \Psi(x_{\ell})]_g}^2\left(2\frac{\alpha_\ell}{L}-\alpha_\ell^2\right)\right\} +(2\alpha_\ell+2\alpha_\ell^2L)\norm{x_{\ell}-x^*}\norm{e_{\B_\ell}(x_\ell)}+\alpha_\ell^2\norm{e_{\B_\ell}(x_\ell)}^2
\end{split}
\end{equation}	

On the other hand, by the definition of $e_{\B}(x)$, we have that 
\begin{equation}\label{eq:error_eq}
\begin{split}
e_{\B}(x)=&[\Grad \Psi_{\B}(x)-\Grad \Psi(x)]_{\I^{\neq 0}(x)}=[\Grad f_{\B}(x)-\Grad f(x)]_{\I^{\neq 0}(x)}\\
=&\frac{1}{|\B|}\sum_{j\in \B}[\Grad f_{j}(x)]_{\I^{\neq 0}(x)}-\frac{1}{N}\sum_{i=1}^N [\Grad f_i(x)]_{\I^{\neq 0}(x)}\\
=&\frac{1}{N}\sum_{j\in\B}\left[\frac{N}{|\B|}[\Grad f_{j}(x)]_{\I^{\neq 0}(x)}-[\Grad f_j(x)]_{\I^{\neq 0}(x)}\right]-\frac{1}{N}\sum_{\substack{i=1\\ i\notin \B}}^N[\Grad f_i(x)]_{\I^{\neq 0}(x)}\\
=&\frac{1}{N}\sum_{j\in\B}\left[\frac{N-|\B|}{|\B|}[\Grad f_{j}(x)]_{\I^{\neq 0}(x)}\right]-\frac{1}{N}\sum_{\substack{i=1\\ i\notin \B}}^N[\Grad f_i(x)]_{\I^{\neq 0}(x)}\\
\end{split}
\end{equation}
Thus taking the norm on both side of~\eqref{eq:error_eq} and using triangle inequality results in the following:
\begin{equation}\label{eq:bound_error}
\begin{split}
\norm{e_\B(x)}&\leq \frac{1}{N}\sum_{j\in\B}\left[\frac{N-|\B|}{|\B|}\norm{[\Grad f_{j}(x)]_{\I^{\neq 0}(x)}}\right]+\frac{1}{N}\sum_{\substack{i=1\\ i\notin \B}}^N\norm{[\Grad f_i(x)]_{\I^{\neq 0}(x)}}\\
&\leq \frac{1}{N} \frac{N-|\B|}{|\B|} |\B_k| M + \frac{1}{N} (N-|\B|)M\leq \frac{2(N-|\B|)M}{N}.
\end{split}
\end{equation}

Since $\alpha_{\ell}\leq 1$, and $|B_\ell|\geq N-\frac{N}{2M}$ hence $\alpha_{\ell}\norm{e_{\B_{\ell}}(x_{\ell})}\leq 1$. Then combining with $\alpha_\ell\leq 1/L$,~\eqref{eq:optimal_dist_ell} can be further simplified as 
\begin{equation}\label{eq:x_kp1_optimal_dist_2}
\begin{split}
&\norm{x_{\ell+1}-x^*}^2\\
\leq & \sum_{g\in\tilde{\G}_\ell}\left\{\norm{[x_{\ell}-x^*]_g}^2-\norm{[\Grad \Psi(x_{\ell})]_g}^2\left(2\frac{\alpha_\ell}{L}-\alpha_\ell^2\right)\right\} +(2\alpha_\ell+2\alpha_\ell^2L)\norm{x_{\ell}-x^*}\norm{e_{\B_\ell}(x_\ell)}+\alpha_\ell^2\norm{e_{\B_\ell}(x_\ell)}^2\\
\leq & \sum_{g\in\tilde{\G}_\ell}\left\{\norm{[x_{\ell}-x^*]_g}^2-\frac{1}{L^2}\norm{[\Grad \Psi(x_{\ell})]_g}^2\right\}+4\alpha_\ell\norm{x_{\ell}-x^*}\norm{e_{\B_\ell}(x_\ell)}+\alpha_\ell^2\norm{e_{\B_\ell}(x_\ell)}^2\\
\leq &\norm{x_\ell-x^*}^2+4\alpha_\ell\norm{x_{\ell}-x^*}\norm{e_{\B_\ell}(x_\ell)}+\alpha_\ell\norm{e_{\B_\ell}(x_\ell)}
\end{split}
\end{equation}
Following from the assumption that $\norm{x_{\ell}-x^*}\leq R$, then~\eqref{eq:x_kp1_optimal_dist_2} can be further simplified as 
\begin{equation}\label{eq:x_kp1_optimal_dist_3}
\begin{split}
\norm{x_{\ell+1}-x^*}^2\leq & \norm{x_\ell-x^*}^2+4\alpha_\ell R\norm{e_{\B_\ell}(x_\ell)}+\alpha_k\norm{e_{\B_\ell}(x_\ell)}\\
\leq & \norm{x_\ell-x^*}^2+(4R+1)\alpha_\ell\norm{e_{\B_\ell}(x_\ell)}
\end{split}    
\end{equation}
Summing the the both side of~\eqref{eq:x_kp1_optimal_dist_3} from $\ell=K$ to $\ell=K+t-1$ results in  
\begin{equation}
\begin{split}
\norm{x_{K+t}-x^*}^2\leq \norm{x_K-x^*}^2+(4R+1)\sum_{\ell=K}^{K+t-1}\alpha_{\ell}\norm{e_{\B_{\ell}}(x_{\ell})}\\
\end{split}
\end{equation}
It follows Lemma~\ref{lemma:convergence-series} that the followng holds with probability at least $1-\frac{1}{\theta^2}$,
\begin{equation}
\sum_{\ell = K}^{\infty} \alpha_{\ell} \|e_{\mathcal{B}_{\ell}}(x_{\ell})\| \leq \frac{3R^2}{4(4R + 1)}.
\end{equation}
Thus we have that 
\begin{equation}
\begin{split}
\norm{x_{K+t}-x^*}^2&\leq \norm{x_K-x^*}^2+\left(4R+1\right)\sum_{\ell=K}^{K+t-1}\alpha_{\ell}\norm{e_{\B_{\ell}}(x_{\ell})}\\
&\leq \norm{x_K-x^*}^2+\left(4R+1\right)\sum_{\ell = K}^{\infty} \alpha_{\ell} \|e_{\mathcal{B}_{\ell}}(x_{\ell})\|\\
&\leq \frac{R^2}{4}+(4R+1)\frac{3R^2}{4(4R+1)}\leq \frac{R^2}{4}+\frac{3R^2}{4}\leq R^2,
\end{split}
\end{equation}
holds with probability at least $1-\frac{1}{\theta^2}$, which completes the proof. 

\end{proof}

Based on the above lemmas, the Lemma~\ref{lemma:x_k_in_neghibors} below shows if initial iterate of~\halfspacestep{} locates closely enough to $x^*$, step size $\alpha_k$ polynomially decreases, and mini-batch size $\B_k$ polynomially increases, then $x^*$ inhabits all subsequent  reduced space $\{\S_k\}_{k=K}^{\infty}$ constructed in~\halfspacestep{} with high probability. 

\begin{lemma}\label{lemma:x_k_in_neghibors}
	Suppose $\norm{x_{K}-x^*}\leq \frac{R}{2}$, $K\geq N_\P$, $k=K+t$, $t\in\mathbb{Z}^+$, $0<\alpha_k=\O(1/(\sqrt{N}t))\leq \min\{\frac{2(1-\epsilon)}{L}, \frac{1}{L},\frac{2\delta_1-R-\epsilon(2\delta_2+R)}{M}\} $ and $|\B_k|=\O(t)\geq N-\frac{N}{2M}$. Then for any constant $\tau\in (0,1)$, $\norm{x_k-x^*}\leq R$ with probability at least $1-\tau$ for any $k\geq K$. 
\end{lemma}
\begin{proof}
	It follows Lemma~\ref{lemma:x_star_in_polyhedron} and the assumption of this lemma that $x^*\in\S_K$. Moreover, it follows the assumptions of this lemma, Lemma~\ref{lemma:convergence-series} and~\ref{lemma:k_plus_1_optimal_dist_non_increase}, the definition of finite-sum $f(x)$ in \eqref{prob.x}, and the bound of error as~\eqref{eq:bound_error} that
	\begin{equation}
	\mathbb{P}(\{x_k\}_{k=K}^{\infty}\in \{x: \norm{x-x^*}\leq R\}^{\infty})\geq \left(1-\frac{1}{\theta^2}\right)^{\O(N-K)}\geq 1-\tau,    
	\end{equation}
	where the last two inequalities comes from that the error vanishing to zero as $|\B_k|$ reaches the upper bound $N$, and $\theta$ is sufficiently large depending on $\tau$ and $\O(N-K)$. 
\end{proof}

\begin{corollary}\label{corollary:x_star_in_all_polyhedrons}
Lemma~\ref{lemma:x_k_in_neghibors} further implies $x^*$ inhabits all subsequent $\S_k$, i.e., $x^*\in \S_{k}$  for any $k\geq K$.
\end{corollary}

Next, we establish that after finitely number of iterations,~\algacro{} generates sequences that inhabits in the feasible domain $\mathcal{X}$ where Lipschitz continuity of $\Psi$ holds. 
\begin{lemma}
Suppose the assumptions of Lemma~\ref{lemma:x_k_in_neghibors} hold, then after finite number of iterations, all subsequent iterates $x_k\in\mathcal{X}$ with high probability.
\end{lemma}
\begin{proof}
It follows Lemma~\ref{lemma:x_k_in_neghibors} that all subsequent $x_k$ satisfying $\norm{x_k-x^*}\leq R$ with high probability. Combining with Lemma~\ref{lemma:support_cover}, we have that $\I^{\neq 0}(x^*)\subseteq\I^{\neq 0}(x_k)$ for all $k\geq K$ with high probability. Then for any $g\in\I^{\neq 0}(x_k)$, there are two possbilities, either $g\in\I^{\neq 0}(x^*)$ or $g\in \I^{0}(x^*)$.  For the first case $g\in\I^{\neq 0}(x^*)\bigcap\I^{\neq 0}(x_k)$, it follows the definitions of $R$ as~\eqref{def:R} and $\delta_1$ that
\begin{equation}
\begin{split}
\norm{[x_k-x^*]_g}&\leq \norm{x_k-x^*}\leq R\leq \delta_1\\
\norm{[x^*]_g}-\norm{[x_k]_g}&\leq \delta_1\\
\norm{[x_k]_g}&\geq \norm{[x^*]_g}-\delta_1\geq2\delta_1-\delta_1=\delta_1
\end{split}
\end{equation}
For any $g\in \I^{0}(x^*)\bigcap\I^{\neq 0}(x_k)$, by~Algorithm~\ref{alg:main.x.halfspacestep}, its norm is bounded below by
\begin{equation}
\delta_1\geq \norm{[x_k-x^*]_g}=\norm{[x_k]_g}\geq \epsilon^t\norm{[x_K]_g},    
\end{equation}
where by the Theorem~\ref{thm:sparsity_recovery_rate_hbproxsg} will shown in  Appendix~\ref{appendix:sparsity_recovery_thm}, if $\norm{[x_k]_g}\leq \frac{2\alpha_k\delta_3}{1-\epsilon+\alpha_kL}$, then $[x_{k+1}]_g$ equals to zero and will be fixed as zero since Algorithm~\ref{alg:main.x.halfspacestep} operates on $\S_k$ as~\eqref{def:polytope}. Note $\alpha_k=\O(1/t)$, following~\cite[Theorem 4]{karimi2016linear} and~\cite[Theorem 3.2]{drusvyatskiy2018error}, $\mathbb{E}[\norm{[x_k]_g}^2]=\O(1/t)$. If $\epsilon>0$, then after finite number of iterations $\O(1/\epsilon^2)$, $g\in \I^{0}(x^*)\bigcap\I^{\neq 0}(x_k)$ becomes zero. If $\epsilon=0$, note $\B_k=\O(t)$ and $f$ is finite-sum, then similar result holds by~\cite[Theorem 2.3, Theorem 3.2]{robert2018convergegd} ($f$ needs further strongly convexity on $\tilde{\mathcal{X}}$). Hence with high probability, after finite number of iterations, denoted by $T$, all subsequent $x_k$, $k\geq K+T$ inhabits $\mathcal{X}$. Regarding $[x_k]_{g\in \I^{0}(x^*)\bigcap\I^{\neq 0}(x_k)}$ for $K\leq k\leq K+T$, note $\epsilon^{t}\norm{[x_K]_g}$ is also bounded below by constant $\epsilon^{T}\norm{[x_K]_g}>0$ given $x_K$, for similicity, denote the Lipschitz constant of $[\Grad \Psi(x_k)]_g$ as $L$ as well. 
\end{proof}
\newpage
We now prove the first main theorem of~\algacro{}.\\
\textbf{Proof of Theorem~\ref{thm:convergence}}
We know that Algorithm~\ref{alg:main.x.outline} performs an infinite sequence of iterations. It follows Corollary~\ref{corollary:Psi_epoch_decrease} that for any $\ell\in\mathbb{Z}^+$, 
\begin{equation}
\begin{split}
&\mathbb{E}[\Psi(x_K)]-\mathbb{E}[\Psi(x_{\ell+1})]=\sum_{k=K}^{\ell}\left\{	\mathbb{E}[\Psi(x_{k})]-\mathbb{E}[\Psi(x_{k+1})]\right\}\\
\geq&\sum_{K\leq k\leq \ell}\left(\alpha_k-\frac{\alpha_k^2L}{2}\right)\sum_{g\in\tilde{\G}_k}\mathbb{E}\left[\norm{[\Grad \Psi(x_{k})]_g}^2\right]+\sum_{K\leq k\leq \ell}\left(\frac{1-\epsilon}{\alpha_k}-\frac{L}{2}\right)\sum_{g\in\hat{\G}_k}\norm{[x_{k}]_g}^2.
\end{split}
\end{equation}
Combining the assumption that $\Psi$ is bounded below and letting $\ell\rightarrow \infty$, we obtain 
\begin{equation}\label{series:sum_convergent_half_space}
\begin{split}
\sum_{k\geq K}\left(\alpha_k-\frac{\alpha_k^2L}{2}\right)\sum_{g\in\tilde{\G}_k}\mathbb{E}\left[\norm{[\Grad \Psi(x_{k})]_g}^2\right]+\sum_{k\geq K}\left(\frac{1-\epsilon}{\alpha_k}-\frac{L}{2}\right)\sum_{g\in\hat{\G}_k}\norm{[x_{k}]_g}^2<\infty
\end{split}
\end{equation}
By~Algorithm~\ref{alg:main.x.halfspacestep}, variables on $\I^{0}(x_k)$ are fixed during~$k$th~\halfspacestep{} and $n$ is finite, then the group projection appears finitely many times, consequently,
\begin{equation}
\sum_{k\geq K}\left(\frac{1-\epsilon}{\alpha_k}-\frac{L}{2}\right)\sum_{g\in\hat{\G}_k}\norm{[x_{k}]_g}^2<\infty.
\end{equation}

Thus~\eqref{series:sum_convergent_half_space} implies that 
\begin{align}\label{series:sum_convergent_half_space_grad_Psi}
&\sum_{k\geq K}\left(\alpha_k-\frac{\alpha_k^2L}{2}\right)\sum_{g\in\tilde{\G}_k}\mathbb{E}\left[\norm{[\Grad \Psi(x_{k})]_g}^2\right]\\
=&\sum_{k\geq K}\alpha_k\sum_{g\in\tilde{\G}_k}\mathbb{E}\left[\norm{[\Grad \Psi(x_{k})]_g}^2\right]- \sum_{k\geq K}\frac{\alpha_k^2}{L}\sum_{g\in\tilde{\G}_k}\mathbb{E}\left[\norm{[\Grad \Psi(x_{k})]_g}^2\right]<\infty
\end{align}
Since $\alpha_k=\O(1/(\sqrt{N}t))$, then $\sum_{k\geq K}\alpha_k=\infty$ and $\sum_{k\geq K}\alpha_k^2\leq\infty$. Combining with~\eqref{series:sum_convergent_half_space_grad_Psi} and the boundness of $\partial \Psi$, it implies  
\begin{equation}\label{series:sum_convergent_half_space_alpha_grad_Psi}
\sum_{k\geq K}\alpha_k\sum_{g\in\tilde{\G}_k}\mathbb{E}\left[\norm{[\Grad \Psi(x_{k})]_g}^2\right]<\infty.    
\end{equation}
By $\sum_{k\geq K}\alpha_k=\infty$ and~\eqref{series:sum_convergent_half_space_alpha_grad_Psi}, we have that 
\begin{equation}
\liminf_{k\geq K} \sum_{g\in\tilde{\G}_k}\mathbb{E}\left[\norm{[\Grad \Psi(x_{k})]_g}^2\right]=0
\end{equation}
then there exists a subsequence $\mathcal{K}$ such that 
\begin{equation}\label{series:sum_convergent_half_space_grad_Psi_subsequence}
\lim_{k\in\mathcal{K}}\sum_{g\in\tilde{\G}_k}\mathbb{E}\left[\norm{[\Grad \Psi(x_{k})]_g}^2\right]=0
\end{equation}
It follows from the assumptions of this theorem and Lemma~\ref{lemma:support_cover} to~\ref{lemma:x_k_in_neghibors} and Corollay~\ref{corollary:x_star_in_all_polyhedrons} that with high probability at least $1-\tau$, for each $k\geq K$, $x^*$ inhabits $\S_k$. Note as $|\B_k|=\O(t)$ linearly increases, the error of gradient estimate vanishes. Hence,~\eqref{series:sum_convergent_half_space_grad_Psi_subsequence} naturally implies that the sequence $\{x_k\}_{k\in\mathcal{K}}$ converges to some stationary point with high probability. And we can extend $\mathcal{K}$ to $\{k: k\geq K\}$ due to the non-decreasing distance to optimal solution as shown in the Lemma~\ref{lemma:x_k_in_neghibors}. By the above, we conclude that 
\begin{equation}
 \mathbb{P}(\lim_{k\rightarrow \infty} \mathbb{E}\left[ \norm{\xi_{\alpha_k,\mathcal{B}_k}(x_k)}\right]=0)\geq 1-\tau.   
\end{equation}

\subsection{Proof of Theorem~\ref{thm:sparsity_recovery_rate_hbproxsg}}\label{appendix:sparsity_recovery_thm}

In this Appendix, we compare the group sparsity identification property of~\algacro{} and~\proxsg{}. We first show the generic sparsity identification property of~\proxsg{} for any mixed $\ell_1/\ell_p$ regularization for $p\geq 1$. 
\begin{lemma}\label{lemma:next_iterate_as_zero1_prox_sg}
	If $\norm{x_{k}-x^*}_{p'}\leq \min\{\delta_3/L, \alpha_k\delta_3\}$, where $1/p+1/p'=1$ $(p'=\infty\ \text{if}\ p = 1)$, then the~\proxsg{} yields that for each $g\in \I^0(x^*)$, $[x_{k+1}]_g = 0$ holds, \ie, $\I^0(x^*)\subseteq \I^0(x_{k+1})$.
\end{lemma}
\begin{proof}
	It follows from the reverse triangle inequality, basic norm inequalities, Lipschitz continuity of $\Grad f(x)$ and the assumption of this lemma that for any $g\in \mathcal{G}$, 
	\begin{equation}\label{eq:grad_f_star_pprime}
	\begin{split}
	\norm{[\Grad f_{\mathcal{B}_k} (x_{k})]_g}_{p'}-\norm{[\Grad f_{\mathcal{B}_k}(x^*)]_g}_{p'}&\leq\norm{[\Grad f_{\mathcal{B}_k}(x_{k})-\Grad f_{\mathcal{B}_k}(x^*)]_g}_{p'}\\
	&\leq \norm{\Grad f_{\mathcal{B}_k}(x_{k})-\Grad f_{\mathcal{B}_k}(x^*)}_{p'}\\
	&\leq L\norm{x_{k}-x^*}_{p'}\leq L\cdot\frac{\delta_3}{L}=\delta_3.
	\end{split}
	\end{equation}
	By~\eqref{eq:grad_f_star_pprime}, we have that for any $g\in \I^{0}(x^*)$, 
	\begin{equation}\label{eq:grad_f_x_k_delta3}
	\begin{split}
	\norm{[\Grad f_{\mathcal{B}_k}(x_{k})]_g}_{p'}&\leq \norm{[\Grad f_{\mathcal{B}_k}(x^*)]_g}_{p'} + \delta_3\\
	&\leq \lambda - 2\delta_3 + \delta_3 = \lambda - \delta_3
	\end{split}
	\end{equation}
	Combining~\eqref{eq:grad_f_x_k_delta3} and the assumption of this lemma, the following holds for any $\alpha_k>0$ that 
	\begin{equation}
	\begin{split}
	\norm{[x_{k}-\alpha_k\Grad f_{\mathcal{B}_k}(x_{k})]_g}_{p'}&\leq \norm{[x_{k}]_g}_{p'} + \norm{[\alpha_k\Grad f_{\mathcal{B}_k}(x_{k})]_g}_{p'}\\
	&\leq \alpha_k\delta_3+\alpha_k(\lambda - \delta_3)=\alpha_k\lambda
	\end{split}
	\end{equation} 
	which further implies that the Ecludiean projection yields that 
	\begin{equation}\label{eq:proj_x_k_grad_f}
	\proj^E_{\mathcal{B}(\norm{\cdot}_{p'},\alpha_k\lambda)} ([x_{k}-\alpha_k\Grad f_{\mathcal{B}_k}(x_k)]_g) = [x_{k}-\alpha_k\Grad f_{\mathcal{B}_k} (x_{k})]_g.
	\end{equation}
	Combining with~\eqref{eq:proj_x_k_grad_f}, the fact that proximal operator is the residual of identity operator subtracted by Euclidean project operator onto the dual norm ball and $[x_{k}]_g=0$ for any $g\in\mathcal{I}^0(x^*)$~\citep{chen2018fast},  we have that 
	\begin{equation}
	\begin{split}
	[x_{k+1}]_g&=\prox_{\alpha_k\lambda \norm{\cdot}_p}([x_{k}-\alpha_k \Grad f_{\mathcal{B}_k}(x_{k})]_g)\\
	&=\left[I-\proj^E_{\mathcal{B}(\norm{\cdot}_{p'},\alpha_k\lambda)}\right]\left[x_{k}-\alpha_k \Grad f_{\mathcal{B}_k}(x_k)\right]_g\\
	&=\left[x_{k}-\alpha_k \Grad f_{\mathcal{B}_k}(x_{k})\right]_g-\left[x_{k}-\alpha_k \Grad f_{\mathcal{B}_k}(x_{k})\right]_g=0,
	\end{split}
	\end{equation}
	consequently $\I^0(x^*)\subseteq \I^0(x_{k+1})$, which completes the proof.
\end{proof}

Now we establish the group-sparsity identification of~\algacro{}. \\\\
\textbf{Proof of Theorem~\ref{thm:sparsity_recovery_rate_hbproxsg}:}

	Suppose $\norm{x_k-x^*}\leq \frac{2\alpha_k\delta_3}{1-\epsilon+\alpha_kL}$. There is nothing to prove if $g\in \I^{0}(x^*)\bigcap \I^{0}(x_k)$. For $g\in \I^0(x^*)\bigcap \I^{\neq 0}(x_k)$,  we compute that
	\begin{equation}\label{eq:tmp_1}
	\begin{split}
	&[x_k-\alpha_k\Grad \Psi_{\B_k}(x_k)]_g^\top[x_k]_g- \epsilon\norm{[x_k]_g}^2\\
	=&\norm{[x_k]_g}^2-\alpha_k[\Grad \Psi_{\B_k}(x_k)]_g^\top[x_k]_g-\epsilon\norm{[x_k]_g}^2\\
	=&(1-\epsilon)\norm{[x_k]_g}^2-\alpha_k\left([\Grad f_{\B_k}(x_k)]_g+\lambda \frac{[x_k]_g}{\norm{[x_k]_g}}\right)^\top[x_k]_g\\
	=&(1-\epsilon)\norm{[x_k]_g}^2-\alpha_k[\Grad f_{\B_k}(x_k)]_g^\top[x_k]_g-\alpha_k\lambda \norm{[x_k]_g}\\
	\leq & (1-\epsilon)\norm{[x_k]_g}^2+\alpha_k\norm{[\Grad f_{\B_k}(x_k)]_g}\norm{[x_k]_g}-\alpha_k\lambda \norm{[x_k]_g}\\
	=& \norm{[x_k]_g}\left\{(1-\epsilon)\norm{[x_k]_g}+\alpha_k\norm{[\Grad f_{\B_k}(x_k)]_g}-\alpha_k\lambda\right\}
	\end{split}
	\end{equation}
	
	By the Lipschitz continuity of $\Grad f$, we have that for each $g\in \I^0(x^*)\bigcap \I^{\neq 0}(x_k)$, 
	\begin{equation}
	\begin{split}
	\norm{[\Grad f_{\B_k}(x_k)-\Grad f_{\B_k}(x^*)]_g}&\leq L\norm{[x_k-x^*]_g}=L\norm{[x_k]_g}\\
	\norm{[\Grad f_{\B_k}(x_k)]_g}&\leq L\norm{[x_k]_g}+\norm{[\Grad f_{\B_k}(x^*)]_g}
	\end{split}
	\end{equation}
	Combining with the definition of $\delta_3$, which implies that $\norm{[\Grad f_{\B_k}(x^*)]_g}\leq \lambda-2\delta_3$ that 
	\begin{equation}
	\norm{[\Grad f_{\B_k}(x_k)]_g}\leq L\norm{[x_k]_g}+\lambda -2\delta_3
	\end{equation}
	Hence combining with $\norm{[x_k]_g}\leq \frac{2\alpha_k\delta_3}{1-\epsilon+\alpha_kL}$,~\eqref{eq:tmp_1} can be further written as 
	\begin{equation}
	\begin{split}
	&[x_k-\alpha_k\Grad \Psi_{\B_k}(x_k)]_g^\top[x_k]_g- \epsilon\norm{[x_k]_g}^2\\
	\leq & \norm{[x_k]_g}\left\{(1-\epsilon)\norm{[x_k]_g}+\alpha_k\norm{[\Grad f_{\B_k}(x_k)]_g}-\alpha_k\lambda\right\}\\
	\leq &  \norm{[x_k]_g}\left\{(1-\epsilon)\norm{[x_k]_g}+\alpha_kL \norm{[x_k]_g} +\alpha_k\lambda-2\alpha_k\delta_3-\alpha_k\lambda\right\}\\
	= &  \norm{[x_k]_g}\left\{(1-\epsilon+\alpha_kL)\norm{[x_k]_g}-2\alpha_k\delta_3\right\}\\
	\leq & \norm{[x_k]_g}\left\{(1-\epsilon+\alpha_kL)\frac{2\alpha_k\delta_3}{1-\epsilon+\alpha_kL}-2\alpha_k\delta_3\right\}\\
	= & \norm{[x_k]_g}\left(2\alpha_k\delta_3-2\alpha_k\delta_3\right)=0.
	\end{split}
	\end{equation}
	which shows that $[x_k-\alpha_k\Grad \Psi_{\B_k}(x_k)]_g^\top[x_k]_g\leq \epsilon\norm{[x_k]_g}^2$. Hence the group projection operator is trigerred on $g$ to map the variables to zero, then $g\in \I^{0}(x_{k+1})$,~\ie, $[x_{k+1}]_g=0$. Therefore, the group sparsity of $x^*$ can be successfully identified by~\halfspacestep{}, \ie, $\I^0(x^*)\subseteq \I^0(x_{k+1})$.\\\\

In the end, if further assumptions hold, we can further show its group-support recovery. 
\begin{corollary}
Under the assumption of Theorem~\ref{thm:sparsity_recovery_rate_hbproxsg}, moreover, if $\norm{x_k-x^*}\leq R$, $x^*\in \S_k$, 
$0\leq \epsilon<\min\left\{\frac{\delta_1^2}{\delta_2}, \frac{2\delta_1-R}{2\delta_2+R}\right\}$ and $\alpha_k \leq\frac{2\delta_1-R-\epsilon(2\delta_2+R)}{M}$, then $\I^0(x^*)= \I^0(x_{k+1})$ and $\I^{\neq 0}(x_{k+1})=\I^{\neq 0}(x^*)$.
\end{corollary}
\begin{proof}
 Moreover, besides $\norm{x_k-x^*}\leq\frac{2\alpha_k\delta_3}{1-\epsilon+\alpha_kL}$, suppose  $\norm{x_k-x^*}\leq R$, $x^*\in \S_k$, $0\leq \epsilon<\min\left\{\frac{\delta_1^2}{\delta_2}, \frac{2\delta_1-R}{2\delta_2+R}\right\}$ and $\alpha_k \leq\frac{2\delta_1-R-\epsilon(2\delta_2+R)}{M}$. Then $x^*\in \S_k$ indicates that $\I^{\neq 0}(x^*)\subseteq\I^{\neq 0}(x_k)$ by the definition of~$\S_k$. It still holds for $x_{k+1}$ by Lemma~\ref{lemma.project_as_zero_group}, \ie, $\I^{\neq 0}(x^*)\subseteq\I^{\neq 0}(x_{k+1})$. Combining with $\I^{0}(x^*)\subseteq \I^{0}(x_k)$, we have that both group-supports and group sparsity of $x^*$ are identified by~\algacro{}, \ie, $\I^{\neq 0}(x^*)= \I^{\neq 0}(x_{k+1})$ and $\I^{0}(x^*)= \I^{0}(x_{k+1})$.
\end{proof}

	\subsection{Upper bound of $N_\P$ under strongly convexity}\label{appendx:upper_bound_n_p}
	
	\begin{proposition} Suppose the following conditions hold:
		\begin{itemize}
			\item (A1) $\mathbb{E}[\nabla f_{\mathcal{B}_k}(\bm{x})] = \nabla f(\bm{x})$.
			\item (A2) there exists a $\sigma > 0$ such that $\mathbb{E}_{\mathcal{B}} [\| \nabla f_{\mathcal{B}}(\bm{x}) - \nabla f(\bm{x})\|^2] \leq \sigma^2$ for any mini-batch $\mathcal{B}$.
			\item (A3) there exists a $\beta \in (0,1)$ such that $0 < \alpha_k < \frac{1 -\beta}{L}$.
			\item (A4) $f$ is $\mu$-strongly convex.
		\end{itemize}
		Set the step-size $\alpha_k = \frac{1}{2\mu\beta k}$, $k_0 = \lceil\max\{1, \frac{1}{2\mu\beta}\} \rceil$. For any $\tau\in (0,1)$, there exists a $N_\P\in\mathbb{Z}^+$ such that  $N_\P\geq\left\lceil\max\left\{ \frac{8 k_0 \mathbb{E}[\|\bm{x}_{k_0} - \bm{x}^{\ast}\|^2]  }{R^2 \tau}, ~~ \frac{8 \sigma^2 \log (N_\P-1)}{\mu^2 \beta^2 R^2 \tau} \right\} \right\rceil$, such that performing Prox-SG $N_\P$ times yields 
		\begin{equation}
		\|\bm{x}_{N_\P} - \bm{x}^{\ast}\|\leq R/2
		\end{equation}
		with probability at least $1-\tau$.
	\end{proposition}
	
	\begin{proof}
		By the conditions (A1, A2, A3), Assumption 3.1 and Theorem 3.2 in~\cite{rosasco2019convergence}, we have for any $k\geq 2$,
		\begin{align}
		\mathbb{E}[\|\bm{x}_{k} - \bm{x}^{\ast}\|^2] \leq \mathbb{E}[\|\bm{x}_{k_0} - \bm{x}^{\ast}\|^2] \left( \frac{k_0}{k} \right) + \frac{\sigma^2}{\mu^2 \beta^2} \frac{\log (k-1)}{k}.
		\end{align}
		Let $\mathbb{E}[\|\bm{x}_{k_0} - \bm{x}^{\ast}\|^2] = s_{k_0}$. For any $\tau \in (0,1)$, there exists a $N_\P\in\mathbb{Z}^+$ satisfying
		\begin{equation}
		N_\P \geq \left\lceil\max\left\{ \frac{8 k_0 s_{k_0} }{R^2 \tau}, ~~ \frac{8 \sigma^2 \log (N_\P-1)}{\mu^2 \beta^2 R^2 \tau} \right\}\right\rceil,
		\end{equation}
		we have 
		\begin{align}
		\mathbb{E}[\|\bm{x}_{N_\P} - \bm{x}^{\ast}\|^2] \leq \frac{R^2 \tau}{4}. 
		\end{align}
		
		Therefore, by Markov inequality, we have that 
		\begin{align}
		\|\bm{x}_{N_\P} - \bm{x}^{\ast}\|^2 \leq \frac{R^2}{4} ~~ \Leftrightarrow ~~ \|\bm{x}_{N_\P} - \bm{x}^{\ast}\| \leq \frac{R}{2}
		\end{align}
		holds with probability at least $1 - \tau$.
		
	\end{proof}

\section{Additional Numerical Experiments}\label{appendix:experiments}

In this section, we provide additional numerical experiments to \textit{(i)} demonstrate the validness of group sparsity identification of HSPG; \textit{(ii)} provide comprehensive comparison to Prox-SG, RDA and Prox-SVRG on benchmark convex problems; and \textit{(iii)} describe more details regarding our non-convex deep learning experiments shown in the main body. 

\subsection{Linear Regression on Synthetic Data}\label{appendix:convex_exp_linear_regression}

We first numerically validate the proposed HSPG on group sparsity identification by linear regression problems with $\ell_1/\ell_2$ regularizations using synthetic data. Consider a data matrix $A\in\mathbb{R}^{N\times n}$ consisting of $N$ instances and the target variable $y\in\mathbb{R}^N$, we are interested in the following problem:
\begin{equation}\label{eq:lr}
    \minimize{x\in\mathbb{R}^n}\ \frac{1}{2N}\|Ax-y\|^2+ \lambda \sum_{g\in \G}\norm{[x]_g}.
\end{equation}
Our goal is to empirically show that HSPG is able to identify the ground truth zero groups with synthetic data.
We conduct the experiments as follows: \textit{(i)} generate the data matrix $A$ whose elements are uniformly distributed among $[-1, 1]$; \textit{(ii)} generate a vector $x^*$ working as the ground truth solution, where the elements are uniformly distributed among $[-1, 1]$ and the coordinates are equally divided into 10 groups ($|\G|=10$); \textit{(iii)} randomly set a number of groups of $x^*$ to be 0 according to a pre-specified group sparsity ratio; \textit{(iv)} compute the target variable $y=Ax^*$; (v) solve the above problem \eqref{eq:lr} for $x$ with $A$ and $y$ only, and then evaluate the Intersection over Union (IoU) with respect to the identities of the zero groups between the computed solution estimate $\hat x$ by HSPG and the ground truth $x^*$.

We test HSPG on \eqref{eq:lr} under different problem settings. For a slim matrix $A$ where $N\ge n$, we test with various group sparsity ratios among $\{0.1,0.3,0.5,0.7,0.9\}$, and for a fat matrix $A$ where $N<n$, we only test with a certain group sparsity value since a recovery of $x^*$ requires that the number of non-zero elements in $x^*$ is bounded by $N$. Throughout the experiments, we set $\lambda$ to be $100/N$, the mini-batch size $|\mathcal{B}|$ to be 64, step size $\alpha_k$ to be 0.1 (constant), and fine-tune $\epsilon$ per problem. Based on a similar statistical test on objective function stationarity~\citep{zhang2020statistical}, we switch to \halfspacestep{} roughly after 30 epoches. Table~\ref{tb:lr} shows that under each setting, the proposed HSPG correctly identifies the  groups of zeros as indicated by $\textrm{IoU}(\hat{x},x^*)=1.0$, which is a strong evidence to show the correctness of group sparsity idenfitication of HSPG. 

\begin{table}[h]
\centering
\caption{Linear regression problem settings and IoU of the recovered solutions by HSPG.}
\label{tb:lr}
\begin{tabular}{c|cccc}
\hline
                                                                      & \quad $N$\quad     &  \quad $n$  \quad   &  \quad Group sparsity ratio of $x^*$    \quad             &  \quad IoU($\hat x,x^*$)  \quad \\ \hline
\multirow{4}{*}{\begin{tabular}[c]{@{}l@{}} Slim $A$ \\  \end{tabular}} & \quad 10000\quad  &  \quad1000 \quad &  \quad\{0.1, 0.3, 0.5, 0.7, 0.9\}  \quad&  \quad1.0  \quad\\ 
                                                                      & \quad 10000\quad  &  \quad2000 \quad &  \quad\{0.1, 0.3, 0.5, 0.7, 0.9\} \quad &  \quad1.0 \quad \\ 
                                                                      & \quad 10000\quad  &  \quad3000 \quad &  \quad\{0.1, 0.3, 0.5, 0.7, 0.9\} \quad &  \quad1.0  \quad\\ 
                                                                    & \quad 10000\quad  &  \quad4000 \quad &  \quad\{0.1, 0.3, 0.5, 0.7, 0.9\}  \quad&  \quad1.0 \quad \\ \hline
\multirow{4}{*}{\begin{tabular}[c]{@{}l@{}}Fat $A$\\ \end{tabular}}      &  \quad200  \quad  &  \quad1000 \quad&  \quad0.9    \quad                        & \quad 1.0 \quad \\  
                                                                      &  \quad300  \quad  &  \quad1000  \quad& \quad 0.8  \quad                          &  \quad1.0  \quad\\ 
                                                                      &  \quad400  \quad  &  \quad1000  \quad&  \quad0.7  \quad                          &  \quad1.0  \quad\\ 
                                                                      &  \quad500 \quad   &  \quad1000 \quad & \quad 0.6  \quad                          &  \quad1.0  \quad\\ \hline
\end{tabular}
\end{table}

\subsection{Logistic Regression}\label{appendix:convex_exp_logistic_regression}

We then focus on the benchmark convex logistic regression problem with the mixed $\ell_1/\ell_2$-regularization given $N$ examples $(d_1, l_1), \cdots, (d_N, l_N)$ where $d_i\in \mathbb{R}^n$ and $l_i \in \{-1, 1\}$ with the form\vspace{-0.1cm}
\begin{equation}\label{def:minimize_logistic_l1}
\small
\minimize{(x; b)\in \R^{n+1}}\ \frac{1}{N}\sum_{i=1}^N \log(1 + e^{-l_i (x^T d_i +b)}) + \lambda \sum_{g\in \G}\norm{[x]_g},
\end{equation}
for binary classification with a bias $b\in\mathbb{R}$. We set the regularization parameter $\lambda$ as $100/N$ throughout the experiments since it yields high sparse solutions and low object value $f$’s, equally decompose the variables into 10 groups to form $\mathcal{G}$, and  test~problem~\eqref{def:minimize_logistic_l1} on 8 standard publicly available large-scale datasets from LIBSVM repository~\citep{chang2011libsvm}
as summarized in Table~\ref{table:datasets}. All convex experiments are conducted on a 64-bit operating system with an Intel(R) Core(TM) i7-7700K CPU $@$ 4.20 GHz and 32 GB random-access memory.

We run the solvers with a maximum number of epochs as $60$.
The mini-batch size $|\mathcal{B}|$ is set to be $\min\{256, \lceil{0.01N\rceil}\}$ similarly to~\citep{yang2019stochastic}. The step size $\alpha_k$ setting follows~[Section 4]\citep{xiao2014proximal}. Particularly, we first compute a Lipschitz constant $L$ as $\max_{i}\norm{d_i}^2/4$, then fine tune and select constant $\alpha_k\equiv\alpha=1/L$ to~\proxsg{} and~\proxsvrg{} since it exhibits the best results. For~\rda{}, the step size parameter $\gamma$ is fined tuned as the one with the best performance among all powers of $10$. For~\algacro{}, we set $\alpha_k$ as the same as~\proxsg{} and \proxsvrg{} in practice.
We set $N_\mathcal{P}$ as $30N/|\mathcal{B}|$ such that \halfspacestep{} is triggered after employing Prox-SG Step 30 epochs similarly to Appendix~\ref{appendix:convex_exp_linear_regression}, and the control parameter $\epsilon$ in~\eqref{def:proj} as 0.05. We select two $\epsilon$'s as $0$ and $0.05$. 
The final objective value $\Psi$ and $f$, and group sparsity in the solutions are reported in  Table~\ref{table:object_Psi_value_convex}-\ref{table:group_sparsity_convex}, where we mark the best values as bold to facilitate the comparison. Furthermore, Figure~\ref{figure:runtime_convex} plots the relative runtime of these solvers for each dataset, scaled by the runtime of the most time-consuming solver. 

Table~\ref{table:group_sparsity_convex} shows that our~\algacro{} is definitely the best solver on exploring the group sparsity of the solutions. In fact,~\algacro{} under $\epsilon=0.05$ performs all the best except \textit{ijcnn1}.~\proxsvrg{} is the second best solver on group sparsity exploration, which demonstrates that the variance reduction techniques works well in convex setting to promote sparsity, but not in non-convex settings. ~\algacro{} under $\epsilon=0$ performs much better than~\proxsg{} which matches the better sparsity recovery property of~\algacro{} as stated in Theorem~\ref{thm:sparsity_recovery_rate_hbproxsg} even under $\epsilon$ as $0$. 
Moreover, as shown in Table~\ref{table:object_Psi_value_convex} and~\ref{table:object_f_value_convex}, we observe that all solvers perform quite competitively in terms of final objective values (round up to 3 decimals) except~\rda{}, which demonstrates that \algacro{} reaches comparable convergence as~\proxsg{} and~\proxsvrg{} in practice.  Finally, Figure~\ref{figure:runtime_convex} indicates that Prox-SG, RDA and \algacro{} have similar computational cost to proceed, except~\proxsvrg{} due to its periodical full gradient computation.

\vspace{-0.2cm}
\begin{table}[h]
	\scriptsize
	\centering
	\def\arraystretch{1.1}
	\caption{Summary of datasets.\label{table:datasets}}
	\begin{tabular}{ccccccccc}
		\Xhline{2\arrayrulewidth}
		Dataset & N & n  & Attribute & & Dataset & N & n  & Attribute \\
		\hline
		a9a & 32561 & 123 & binary \{0, 1\} & & news20 & 19996 & 1355191 &   unit-length \\
		higgs & 11000000 & 28 & real $[-3, 41]$ & & real-sim & 72309 & 20958 & real [0, 1]\\
		ijcnn1 & 49990 & 22  &  real [-1, 1] & &  url\_combined & 2396130 & 3231961 & real $[-4, 9]$ \\
		kdda & 8407752 & 20216830 & real $[-1, 4]$ & & w8a & 49749 & 300   & binary \{0, 1\}\\
		\Xhline{2\arrayrulewidth}
	\end{tabular}
\end{table}

\begin{table}[h]
	
	\centering
	    \def\arraystretch{1.1}

	    \caption{Final objective values $\Psi$ for tested algorithms on convex problems.}
	    \label{table:object_Psi_value_convex}
		{\scriptsize
		\begin{tabularx}{\textwidth} { 
           >{\centering\arraybackslash}X 
           >{\centering\arraybackslash}X 
           >{\centering\arraybackslash}X 
           >{\centering\arraybackslash}X 
           >{\centering\arraybackslash}X 
           >{\centering\arraybackslash}X  }
			\Xhline{3\arrayrulewidth}
            \multirow{2}{*}{Dataset} & \multirow{2}{*}{\proxsg{}} & \multirow{2}{*}{\rda} & \multirow{2}{*}{\proxsvrg{}}  & \multicolumn{2}{c}{\algacro{}} \\
            \cline{5-6}
            & &  & &  $\epsilon$ as $0$ & $\epsilon$ as $0.05$\\
			\hline
			a9a & \textbf{0.355} & 0.359  & \textbf{0.355} & \textbf{0.355} & \textbf{0.355} \\
			higgs & \textbf{0.357} & 0.360 & 0.365 & 0.358 & 0.358\\
			ijcnn1 & \textbf{0.248} & 0.278 & \textbf{0.248} & \textbf{0.248} & \textbf{0.248}\\
			kdda & \textbf{0.103} & 0.124 & \textbf{0.103} & \textbf{0.103} & \textbf{0.103}\\
			news20 & \textbf{0.538} & 0.693 & \textbf{0.538} & \textbf{0.538} & \textbf{0.538}  \\
			real-sim & \textbf{0.242} & 0.666 & 0.244 & \textbf{0.242} & \textbf{0.242} \\
			url\_combined & 0.397 & 0.579  & \textbf{0.391} &  0.405 & 0.405\\
			w8a & \textbf{0.110} & 0.111 & 0.112 & \textbf{0.110} & \textbf{0.110}\\
			\Xhline{3\arrayrulewidth} 
		\end{tabularx}
		}
	    \caption{Final objective values $f$ for tested algorithms on convex problems.}
	    \label{table:object_f_value_convex}

		{\scriptsize
		\begin{tabularx}{\textwidth} { 
           >{\centering\arraybackslash}X 
           >{\centering\arraybackslash}X 
           >{\centering\arraybackslash}X 
           >{\centering\arraybackslash}X 
           >{\centering\arraybackslash}X 
           >{\centering\arraybackslash}X  }
			\Xhline{3\arrayrulewidth}
            \multirow{2}{*}{Dataset} & \multirow{2}{*}{\proxsg{}} & \multirow{2}{*}{\rda} & \multirow{2}{*}{\proxsvrg{}} & \multicolumn{2}{c}{\algacro{}} \\
            \cline{5-6}
            & &  & &  $\epsilon$ as $0$ & $\epsilon$ as $0.05$\\
			\hline
			a9a & \textbf{0.329} & 0.338  & \textbf{0.329} & \textbf{0.329} & \textbf{0.329} \\
			higgs & \textbf{0.357} & 0.360 & 0.365 & 0.358 & 0.358 \\
			ijcnn1 & \textbf{0.213} & 0.270 & \textbf{0.213} & \textbf{0.213} & 0.214\\
			kdda & \textbf{0.103} & 0.124 & \textbf{0.103} & \textbf{0.103} & \textbf{0.103}\\
			news20 & 0.373 & 0.693 & 0.381 & \textbf{0.372} & \textbf{0.372}  \\
			real-sim & \textbf{0.148} & 0.665 & 0.159 & \textbf{0.148} & \textbf{0.148} \\
			url\_combined & 0.397 & 0.579 & \textbf{0.391} & 0.405 & 0.405\\
			w8a & \textbf{0.089} & 0.098 & 0.091 & \textbf{0.089} & \textbf{0.089}\\
			\Xhline{3\arrayrulewidth} 
		\end{tabularx}
		}
		\caption{Group sparsity for tested algorithms on convex problems.}
		\label{table:group_sparsity_convex}
		
		{\scriptsize
		\begin{tabularx}{\textwidth} { 
           >{\centering\arraybackslash}X 
           >{\centering\arraybackslash}X 
           >{\centering\arraybackslash}X 
           >{\centering\arraybackslash}X 
           >{\centering\arraybackslash}X 
           >{\centering\arraybackslash}X }
			\Xhline{3\arrayrulewidth}
            \multirow{2}{*}{Dataset} & \multirow{2}{*}{\proxsg{}} & \multirow{2}{*}{\rda{}} & \multirow{2}{*}{\proxsvrg{}} & \multicolumn{2}{c}{\algacro{}} \\
            \cline{5-6}
            & &  & &  $\epsilon$ as $0$ & $\epsilon$ as $0.05$\\
			\hline
			a9a & 20\% & \textbf{30\%}  & \textbf{30\%} & \textbf{30\%} & \textbf{30\%} \\
			higgs & 0\% & 10\% & 0\% & 0\% & \textbf{30\%}\\
			ijcnn1 & 50\% & \textbf{70\%} & 60\% & 60\% & 60\% \\
			kdda & 0\% & 0\% & 0\% & 0\% & \textbf{80\%}\\
			news20 & 20\% & 80\% & \textbf{90\%} & 80\% & \textbf{90\%}  \\
			real-sim & 0\% & 0\% & \textbf{80\%} & 0\% & \textbf{80\%} \\
			url\_combined & 0\% & 0\% & 0\% & 0\% & \textbf{90\%} \\
			w8a & \textbf{0\%} & \textbf{0\%} & \textbf{0\%} & \textbf{0\%} & \textbf{0\%} \\
			\Xhline{3\arrayrulewidth} 
		\end{tabularx}
		}
\end{table} 

\begin{figure}
    \centering
	\includegraphics[width=\textwidth]{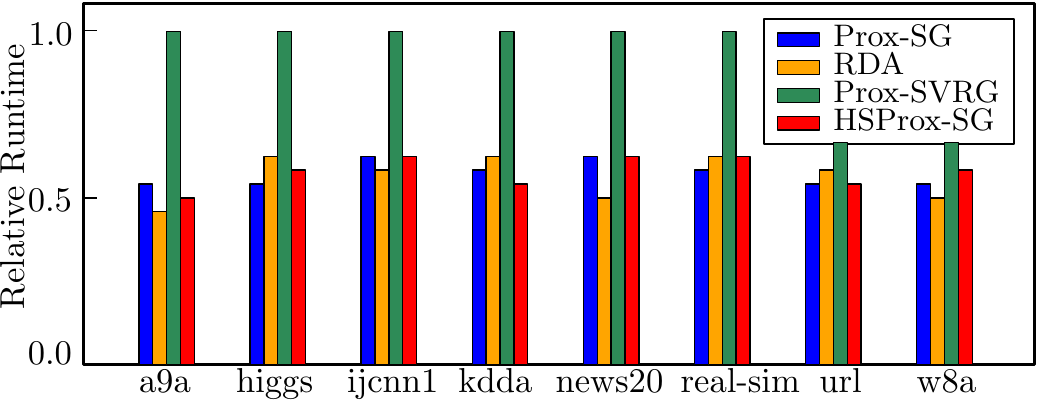}
	\caption{Relative runtime.}
	\label{figure:runtime_convex}
\end{figure}

\newpage

\subsection{Deep Learning Experiments}\label{appendix:nonconvex_exp}

We conduct all deep learning experiments on one GeForce GTX 1080 Ti GPU, and describe how to fine-tune the control parameter $\epsilon$ in~\eqref{def:proj} in details. According to Theorem~\ref{thm:sparsity_recovery_rate_hbproxsg}, a larger $\epsilon$ results in a faster group sparsity identification, while by Lemma~\ref{lemma:sufficient_decrease_half_space} on the other hand too large $\epsilon$ may cause a significant regression on the target objective $\Psi$ value, \ie, the $\Psi$ value increases a lot. Hence, in our experiments, from the point of view of optimization, we search a proper $\epsilon$ in the following ways: start from $\epsilon=0.0$ and the models trained by employing $N_\P$~\proxsgstep{}s, incrementally increase $\epsilon$ by 0.01 and check if the $\Psi$ on the first~\halfspacestep{} has an obvious increase, then accept the largest $\epsilon$ without regression on $\Psi$ as our fine tuned $\epsilon$ shown in the main body of the paper. Particularly, the fine tuned $\epsilon$'s equal to 0.03, 0.05, 0.02 and 0.02 for \vgg{} with~\cifar{}, \vgg{} with~\fashionmnist{},~\resnet{} with~\cifar{} and~\resnet{} with~\fashionmnist{} respectively. Note from the perspective of different applications, there are different criterions to fine tune $\epsilon$, \ie, for model compression, we may accept $\epsilon$ based on the validation accuracy regression to reach higher group sparsity. 

Additionally, we also report the final $f$ comparison in Table~\ref{table:object_f_value_nonconvex} and its evolution on~\resnet{} with~\cifar{} in Figure~\ref{figure:f_evolution}, where we can see that all tested algorithms can achieve competitive $f$ values as they do in convex settings. And the evolution of $f$ is similar to that of $\Psi$, \ie, the raw objective $f$ generally monotonically decreases for small $\epsilon=0$ to $0.02$, and experiences a mild pulse after switch to~\halfspacestep{} for larger $\epsilon$, \eg, 0.05, which matches Lemma~\ref{lemma:sufficient_decrease_half_space_restated}.
\begin{table}[h]
	\centering
	    \def\arraystretch{1.1}

	    \caption{Final objective values $f$ for tested algorithms on non-convex problems.}
	    \label{table:object_f_value_nonconvex}

		{\scriptsize
		\begin{tabularx}{\textwidth} { 
           >{\centering\arraybackslash}X 
           >{\centering\arraybackslash}X 
           >{\centering\arraybackslash}X 
           >{\centering\arraybackslash}X 
           >{\centering\arraybackslash}X 
           >{\centering\arraybackslash}X  }
			\Xhline{3\arrayrulewidth}
            \multirow{2}{*}{Backbone} & \multirow{2}{*}{Dataset} & \multirow{2}{*}{\proxsg{}} & \multirow{2}{*}{\proxsvrg{}} & \multicolumn{2}{c}{\algacro{}} \\
            \cline{5-6}
            &  & &  & $\epsilon$ as $0$ & fine tuned $\epsilon$\\
			\hline
			\multirow{2}{*}{\vgg{}}& \cifar{} & 0.010 & 0.036 & 0.010 & \textbf{0.009} \\
			& \fashionmnist{} & 0.181 & \textbf{0.165} & 0.181 & 0.182\\
			\hdashline
			\multirow{2}{*}{\resnet{}} & \cifar{} & \textbf{0.001} & 0.002 & \textbf{0.001} & 0.004  \\
			 &\fashionmnist{} & 0.006 & 0.008 & \textbf{0.005} & 0.010 \\
			 \hdashline
			\multirow{2}{*}{\mobilenet{}} & \cifar{} & \textbf{0.021} & 0.031 & \textbf{0.021} & 0.031  \\
			 &\fashionmnist{} & 0.074 & \textbf{0.057} & 0.074 & 0.088 \\
			\Xhline{3\arrayrulewidth} 
		\end{tabularx}
		}
\end{table} 
\begin{figure}[h]
    \centering
    \includegraphics[width=0.5\textwidth]{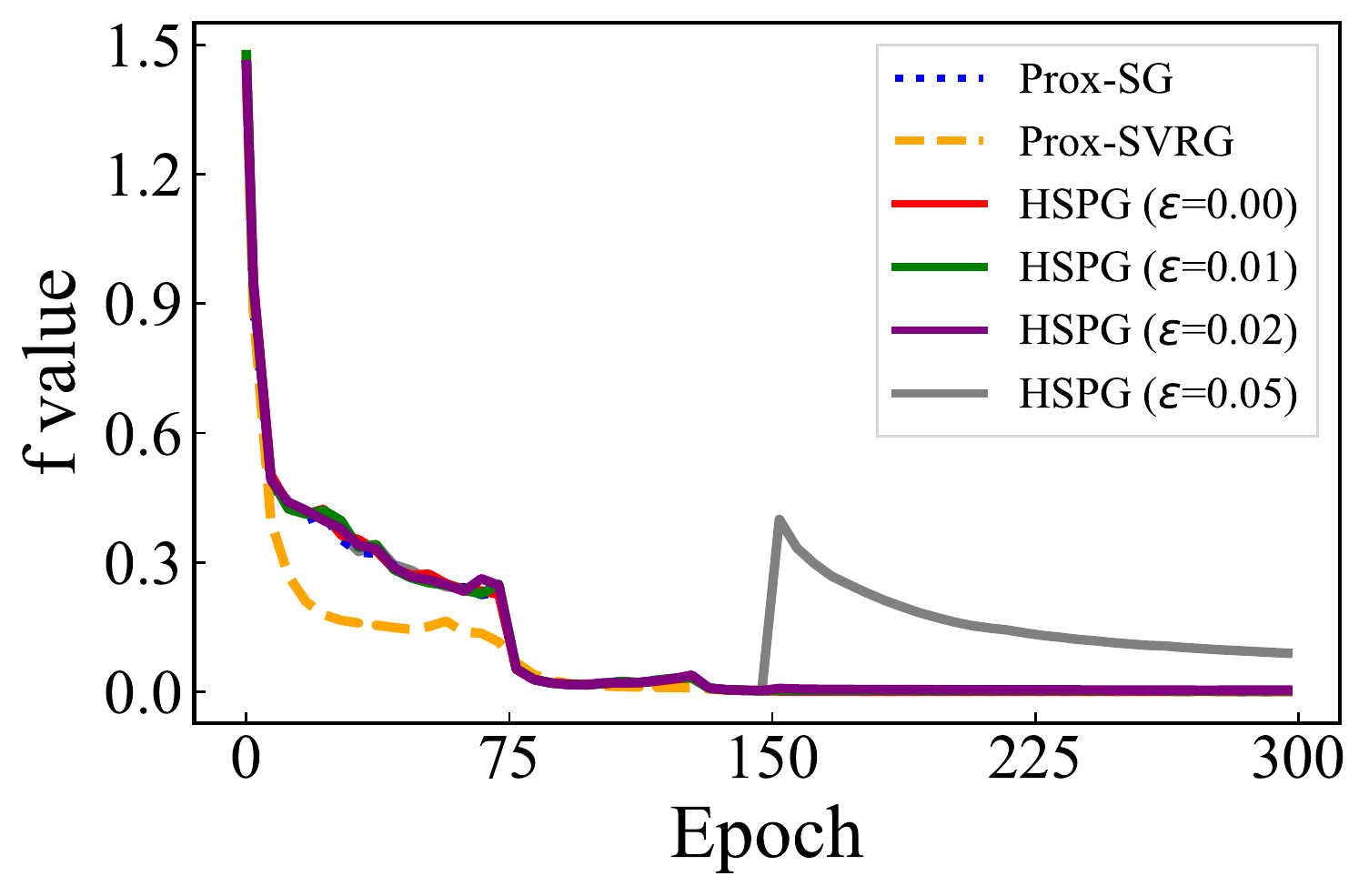}
    \caption{Evolution of $f$ value on~\resnet{} with~\cifar{}.}
    \label{figure:f_evolution}
\end{figure}

\end{document}